\documentclass{amsart}

\usepackage{cancel}
\usepackage{soul}
\usepackage[normalem]{ulem}

\usepackage{amssymb,amsmath,graphicx}
\usepackage{xcolor}
\usepackage{enumerate}
\newtoks\prt

\newtheorem{thm}{Theorem}[section]

\newtheorem{lemma}[thm]{Lemma}
\newtheorem{prop}[thm]{Proposition}
\newtheorem{cor}[thm]{Corollary}

\newtheorem{obs}[thm]{Observation}

\theoremstyle{definition}

\newtheorem{remark}[thm]{Remark}

\def\eqn#1$$#2$${\begin{equation}\label#1#2\end{equation}}
 \def\J#1#2#3{ \left\{ #1,#2,#3 \right\} }

\def\1{\boldsymbol{1}}
\def\A{\mathcal A}

\def\w{\widehat}

\def\ce{\mathbb C}

\def\en{\mathbb N}
\def\er{\mathbb R}

\def\O{\mathbb O}
\def\Ha{\mathbb H}

\def\r{\boldsymbol{r}}

\def \a{\boldsymbol{a}}
\def \b{\boldsymbol{b}}

\def \e{\boldsymbol{e}}
\def \x{\boldsymbol{x}}
\def \y{\boldsymbol{y}}

\def \uu{\boldsymbol{u}}
\def \vv{\boldsymbol{v}}
\def \w{\boldsymbol{w}}
\def \p{\boldsymbol{p}}
\def \q{\boldsymbol{q}}
\def \1{\boldsymbol{1}}

\def \dt {\operatorname{dt}}

\def\span{\operatorname{span}}

\def \reg {\partial _{\kern1pt\text{reg}}}

\def\inv{\diamond}

\def\ip#1#2{\left\langle#1,#2\right\rangle}

\newcommand{\norm}[1]{\left\|#1\right\|}

\newcommand{\abs}[1]{\left|#1\right|}
\newcommand{\setsep}{;\,}

\definecolor{green}{rgb}{0,0.5,0}

\title{Determinants in Jordan matrix algebras}

\author[J. Hamhalter]{Jan Hamhalter}

\author[O.F.K. Kalenda]{Ond\v{r}ej F.K. Kalenda}

\author[A.M. Peralta]{Antonio M. Peralta}

\address[J. Hamhalter]{Czech Technical University in Prague, Faculty of Electrical Engineering, Department of Mathematics, Technicka 2, 166 27, Prague 6,
Czech Republic}
\email{hamhalte@fel.cvut.cz}
\address[O.F.K. Kalenda]{Charles University, Faculty of Mathematics and Physics, Department of
Mathematical Analysis, Sokolovsk{\'a} 86, 186 75 Praha 8, Czech Republic}
\email{kalenda@karlin.mff.cuni.cz}
\address[A.M. Peralta]{Instituto de Matem{\'a}ticas de la Universidad de Granada (IMAG), Departamento de An{\'a}lisis Matem{\'a}tico, Facultad de
Ciencias, Universidad de Granada, 18071 Granada, Spain.}
\email{aperalta@ugr.es}



\begin{document}

\begin{abstract}
We introduce a natural notion of determinant in matrix JB$^*$-algebras, i.e., for hermitian matrices of biquaternions and for hermitian $3\times 3$ matrices of complex octonions. We establish several properties of these determinants which are useful to understand the structure of the Cartan factor of type $6$. As a tool we provide an explicit description of minimal projections in the Cartan factor of type $6$ and a variety of its automorphisms. 
\end{abstract}

\maketitle

\section{Introduction}

Exceptional Cartan factors  are still enigmatic mathematical models for which several aspects of their sets of tripotents remain open. The so-called Cartan factor of type 6 is an example of JB$^*$-algebra, formed by hermitian $3\times 3$ matrices with entries in the algebra of complex octonions, which cannot be represented as a Jordan $^*$-subalgebra of any C$^*$-algebra. Despite the exceptional algebraic and analytic studies by Loos \cite{loos1977bounded}, more concrete descriptions of the tripotents  and the natural  relations among them are aspects which still admit certain margin of improvement.

A \emph{Jordan algebra} is a (real or complex) possibly non-associative algebra $B$ whose product (denoted usually by $\circ$) satisfies two axioms:
\begin{enumerate}[$(i)$]
    \item $\forall x,y\in B\colon x\circ y=y\circ x$ (commutativity);
    \item $\forall x,y\in B\colon x^2\circ(y\circ x)=(x^2\circ y)\circ x$ (the Jordan axiom).
\end{enumerate}
A canonical example is provided by any associative algebra $A$  equipped with the special Jordan product
\begin{equation}\label{eq:special Jordan product}
    x\circ y=\frac12(xy+yx),\quad x,y\in A.\end{equation}
More generally, any linear subspace of an associative algebra which is closed under the just defined Jordan product is a Jordan algebra. Such Jordan algebras are called \emph{special}. Jordan algebras which are not special, i.e., which cannot be embedded as a Jordan subalgebra of an associative algebra, are called \emph{exceptional}.

If a Jordan algebra $B$ is equipped by a complete norm satisfying
$$\forall x,y\in B\colon \norm{x\circ y}\le \norm{x}\cdot \norm{y},$$
it is called a \emph{Jordan Banach algebra}.

There are two closely related classes of Jordan Banach algebras -- a class of real ones called 
\emph{JB-algebras} and a class of complex ones called \emph{JB$^*$-algebras}. 

A JB-algebra is a real Jordan Banach algebra $B$ which additionally satisfies 
$$\forall x,y\in B\colon \norm{x}^2\le \norm{x^2+y^2};$$
while a JB*-algebra is a complex Jordan Banach algebra $B$ equipped with an involution $^*$ satisfying
$$\forall x\in B\colon \norm{x}^3=\norm{2(x^*\circ x)\circ x-x^2\circ x^*}.$$
These two classes are interrelated -- the self-adjoint part of a JB$^*$-algebra is a JB-algebra and, conversely, any JB-algebra is the self-adjoint part of some JB$^*$-algebra (cf. \cite{Wright1977}).

Any C$^*$-algebra becomes a JB$^*$-algebra if we equip it by the special Jordan product \eqref{eq:special Jordan product}. More generally, any closed subspace of a C$^*$-algebra which is stable under involution and the special Jordan product is a JB$^*$-algebra. Such algebras are called JC$^*$-algebras.

An important subclass of JB-algebras are Jordan matrix algebras studied for example in \cite[Section 2.7]{hanche1984jordan} from the algebraic point of view, it follows from \cite[Proposition 2.9.2 and Corollary 3.1.7]{hanche1984jordan} that they are JB-algebras under the natural norm. We will focus on matrix JB$^*$-algebras, which are the complex versions of the mentioned algebras. 

More specifically, we will deal with JB$^*$-algebras $H_n(\Ha_C)$, of all $n\times n$ hermitian matrices of biquaternions, and $H_3(\O)$, of all $3\times 3$ hermitian matrices of (complex) octonions. Note that $H_3(\O)$ is known as the \emph{Cartan factor of type 6} (denoted by $C_6$). It is an exceptional JB$^*$-algebra and, moreover, a prototype of such algebras. More precisely, $H_3(\O)$ is a quotient of an ideal in any exceptional JB$^*$-algebra (cf. \cite[Lemma 7.2.2 and Theorem 7.2.3]{hanche1984jordan} for a real version of this result).

We introduce and investigate determinants of elements of the mentioned matrix JB$^*$-algebras. These determinants have surprisingly many common properties with the classical determinants of complex matrices even though the standard formulae for determinants cannot be directly used.
Original motivation for this research comes from a deeper investigation of unitary elements in the Cartan factor of type 6, but we find the theory of determinants also interesting in itself.

Let us briefly describe the  contents of the paper:

In Section~\ref{sec:2} we provide some background information on JB$^*$-algebras and JB$^*$-triples. In Section~\ref{sec:CD-C6} we recall the Cayley-Dickson doubling process, the  construction of quaternions, biquaternions, (real and complex) octonions and of the Cartan factor of type 6. We also collect basic properties and representations of these structures.

In Section~\ref{sec:unitaries} we introduce determinants of unitary elements of $C_6$ (using their spectral decomposition) and formulate the main structure results on unitaries.
Our main motivation was originally the statement of Theorem~\ref{T:product}, which will be applied in a forthcoming paper to investigate certain relations on tripotents. We also point out Theorem~\ref{T:simultaneous biq}
which enables us to reduce many problems from matrices whose entries are octonions to matrices of biquaternions. These theorems are proved in the later sections.

In Section~\ref{sec:dt-n} we introduce the determinant of hermitian matrices of biquaternions of arbitrary order. It is done by an inductive formula. We also establish several properties of these determinants.

Section~\ref{sec:auto-oct} cointains some auxilliary results on automorphisms of the algebra of octonions. These results are applied in Section~\ref{sec:auto-c6} to get results on automorphism of the Cartan factor of type $6$.

In Section~\ref{sec:projections} we provide an explicit description of minimal projections in $C_6$.

In Section~\ref{sec:proofs} we prove the theorems formulated in Section~\ref{sec:unitaries} using the results of the four preceding sections.

Finally, in Section~\ref{sec:dt in C6} we extend the definition of the determinant to general elements of $C_6$ and establish its basic properties.

\section{Background}\label{sec:2}

In this section we provide some background information on JB$^*$-algebras and JB$^*$-triples which will be used throughout the paper. 

A  \emph{JB$^*$-triple} is a complex Banach space $E$  equipped with a (continuous) triple product $\{.,.,.\}: E^3 \to E$, which is symmetric and bilinear in the outer variables and conjugate-linear in the middle one, and satisfies the following algebraic--analytic axioms
(where given $a,b\in E$, $L(a,b)$ stands for the (linear) operator on $E$ given by $L(a,b)(x)=\J abx$, for all $x\in E$):
\begin{enumerate}[$(JB^*$-$1)$]
\item $L(x,y)L(a,b) = L(L(x,y)(a),b) $ $- L(a,L(y,x)(b))$ $+ L(a,b) L(x,y),$ for all $a,b,x,y\in E$; \hfill \emph{(Jordan identity)}
\item The operator $L(a,a)$ is a hermitian operator with nonnegative spectrum for each $a\in E$;
\item $\|\{a,a,a\}\|=\|a\|^3$ for $a\in E$.
\end{enumerate} 

Given $a\in E,$ the mapping $x\mapsto \{a,x,a\}$ will be denoted by $Q(a).$

Any C$^*$-algebra  $A$ becomes a JB$^*$-triple when equipped with the triple product
\begin{equation}\label{eq triple product for Cstar algebras} \{x,y,z\}=\frac12 (x y^* z+z y^* x)  \ \ (x,y,z\in A).\end{equation}
More generally, any closed subspace of a C$^*$-algebra which is stable under this triple product is a JB$^*$-triple (cf. \cite{harris1974bounded}). These particular types of JB$^*$-triples are called JC$^*$-triples. 

Further, every JB$^*$-algebra $B$ is a JB$^*$-triple under the triple product defined by \begin{equation}\label{eq triple product JBstar algebras} \J xyz = (x\circ y^*) \circ z + (z\circ y^*)\circ x - (x\circ z)\circ y^*
\end{equation} for all $x,y,z\in B$ (see \cite[Theorem 3.3]{braun1978holomorphic} or \cite[Lemma 3.1.6]{chubook},  \cite[Theorem 4.1.45]{Cabrera-Rodriguez-vol1}).

For each element $a$ in a Jordan algebra $B$ the symbol $U_a$ will stand for the linear mapping on $B$ defined by $U_a (x) = 2(a\circ x)\circ a - a^2 \circ x.$ The connection with the triple product given in \eqref{eq triple product JBstar algebras} affirms that $U_a(x^*) = \{ a, x, a\} = Q(a) (x)$ for all $a,x\in B.$ Further, 
if a C$^*$-algebra $A$ is regarded as a JB$^*$-algebra with respect to its natural Jordan product, we have $U_a (x) = a x a,$ for all $a,x\in A.$ 

A Jordan $^*$-homomorphism between JB$^*$-algebras is a linear mapping preserving Jordan products and involution. A triple homomorphism between JB$^*$-triples $E$ and $F$ is a linear mapping $\Phi :E\to F$ preserving triple products. A property increasing the attractiveness of JB$^*$-triples is a thoerem by Kaup showing that a linear bijection between two JB$^*$-triples is a triple isomorphism if and only if it is an isometry (see \cite[Proposition 5.5]{kaup1983riemann}). Let $\Phi : B_1\to B_2$ be a triple homomorphism between two unital JB$^*$-algebras. It is known that   
\begin{equation}\label{eq unital triple homomorphisms are algebra homomorphisms}\hbox{if $\Phi$ is unital (i.e. $\Phi(\1) = \1$), then $\Phi$ is a Jordan $^*$-homomorphism. }
\end{equation} This can be easily deduced from the fact that $\{a,\1,b\} = a\circ b$ and $a^* =\{\1,a,\1\}$, for all $a,b\in B_j$ (cf. the expression in \eqref{eq triple product JBstar algebras}). 

A {\em JBW$^*$-triple} is a JB$^*$-triple which is also a dual Banach space. A generalization of Sakai's theorem, proved by T. Barton and R. Timoney in \cite{BaTi}, asserts that each JBW$^*$-triple admits a unique (isometric) predual and its triple product is separately weak$^*$ continuous. 

Similarly, a {\em JBW$^*$-algebra} is a JB$^*$-algebra which is also a dual Banach space. Again, each JBW$^*$-algebra admits a unique (isometric) predual, the involution is weak$^*$-to-weak$^*$ continuous and the Jordan product is separately weak$^*$-to-weak$^*$ continuous. 

In the present paper we will work only with finite-dimensional JB$^*$-algebras. They are obviously dual spaces, hence the theory of JBW$^*$-algebras (and JBW$^*$-triples) may be applied.

JBW$^*$-triples may be represented in a rather explicit form, see \cite{horn1987classification,horn1988classification}. An important role in this representation is played by the so-called \emph{Cartan factors}, which are rather
concrete and illustrative examples of JBW$^*$-triples.
There are six types of Cartan factors (see, e.g., \cite[Section 7.1.1]{Cabrera-Rodriguez-vol2}). In this note we will work only with Cartan factors of types $4$ and $6$.

A Cartan factor of type 4 (also known as a \emph{spin factor}) is a JB$^*$-triple given by a complex Hilbert space (with inner product $\langle \cdot,\cdot \rangle$) equipped with a conjugation $x \mapsto \overline{x}$, triple product $$\J x y z
= \langle x , y \rangle z + \langle z , y \rangle x - \langle x ,
\overline{z} \rangle  \overline{y},$$ and norm given by $\| x\|^2=\langle x , x
\rangle+\sqrt {\langle x , x \rangle^2-|\langle x , \overline x
\rangle|^2}$.

The Cartan factor of type $6$ is the $27$-dimensional JB$^*$ algebra $H_3(\O)$ of hermitian $3\times 3$ matrices with entries in the
eight-dimensional complex algebra $\O$ known as the Cayley numbers or octonions (note that $\O$ is the eight-dimensional spin factor with an additional structure, see subsections \ref{subsect: real and complex octonions} and \ref{subsec: H3O} for detailed presentations).

A projection in a JB$^*$-algebra is a self-adjoint idempotent, i.e., an element $p$ satisfying $p^*=p$ and $p\circ p=p$. There is a natural partial order on projections ($p\le q$ if $p\circ q=p$). In a JBW$^*$-algebra projections form a complete lattice.

The notion of projection makes no sense in the setting of JB$^*$-triples. An appropriate alternative can be found in the concept of tripotent. An element $e$ in a JB$^*$-triple $E$ is said to be a \emph{tripotent} if $e= \J eee$. Every projection in a C$^*$-algebra $A$ is a tripotent when the latter is regarded as a JB$^*$-triple. However the set of tripotents in $A$ is, in general, bigger as it coincides with the collection of all partial isometries in $A$. 

Different properties of tripotents can be defined in terms of the Peirce decomposion that each of them defines. Concretely, if we fix a tripotent $e$ in a JB$^*$-triple $E$, the whole space decomposes in the form $$E= E_{2} (e) \oplus E_{1} (e)\oplus E_0 (e),$$ where, for $i\in \{0,1,2\},$ $E_i(e)$ is the eigenspace of the operator $L(e, e)$ corresponding to the eigenvalue $\frac{i}{2}.$ This decomposition is known as the \emph{Peirce decomposition} of $E$ relative to $e$ (cf. \cite{Friedman-Russo}, \cite[Definition 1.2.37]{chubook} or \cite[\S 4.2.2, \S 5.7]{Cabrera-Rodriguez-vol1, Cabrera-Rodriguez-vol2}).

If $A$ is a C$^*$-algebra, regarded as a JB$^*$-triple,  and $e\in E$ a tripotent (i.e., a partial isometry) the Peirce decomposition is given by $$A_2(e) = ee^* A e^* e, A_1(e) = (1-ee^*) A e^*e 
\oplus ee^* A (1-e^* e),$$ and $A_0(e) =  (1-ee^*) A  (1-e^* e).$ 

The Peirce-2 subspace, $E_2 (e),$ associated with a tripotent $e$ in a JB$^*$-triple $E$ is a unital JB$^*$-algebra with unit $e$, Jordan product $a\circ_{e} b := \{ a,e,b\}$ and involution $a^{*_e} := \J eae$ (cf. \cite[\S 1.2 and Remark 3.2.2]{chubook}  or \cite[Corollary~4.2.30]{Cabrera-Rodriguez-vol1}). This should be compared with the fact that every JB$^*$-algebra is a JB$^*$-triple with respect to the product \eqref{eq unital triple homomorphisms are algebra homomorphisms}.
By Kaup's theorem (see \cite[Proposition 5.5]{kaup1983riemann}) the triple product on $E_2 (e)$ is uniquely determined by the expression
\begin{equation}\label{eq product Peirce2 as JB*-algebra} \{ a,b,c\} =(a \circ_{e} b^{*_e}) \circ_{e} c +(c \circ_{e} b^{*_e}) \circ_e a - (a \circ_e c) \circ b^{*_e},
\end{equation}
for every $a,b,c\in E_2 (e)$. Therefore, unital JB$^*$-algebras are in one-to-one correspondence with JB$^*$-triples
admitting a unitary element. 

A tripotent $e$ in $E$ is called \emph{complete} (respectively \emph{minimal}) if $E_0 (e) =\{0\}$ (respectively, $E_2(e) = \mathbb{C} e\neq \{0\}$). If $E= E_2(e)$, or equivalently, if $\{e,e,x\}={x}$ for all $x\in E$, we say that $e$ is \emph{unitary}. We recall that two tripotents $e,v$ in $E$ are \emph{orthogonal} (denoted by $e\perp v$) if $\{e,e,v\}=0$ ($\Leftrightarrow$ $\{v,v,e\}=0$ $\Leftrightarrow$ $e\in E_0(v)$ $\Leftrightarrow$ $v\in E_0(e)$).  We shall say that $e$ is a \emph{finite-rank tripotent} if $e$ can be written as a finite sum of mutually orthogonal minimal tripotents in $E$. The relation of orthogonality can be employed to define a partial order on the set of tripotents given by $e\leq u$ if $u-e$ is a tripotent orthogonal to $e$. There are several equivalent reformulations of this partilal order (cf., for example, those gathered in \cite[\S 6.3]{hamhalter2019mwnc}), one of then says that $e\leq u$ if and only if $e$ is a projection in the JB$^*$-algebra $E_2(u)$.  

There is a wider notion of orthogonality for general elements in a JB$^*$-triple $E$. Namely, elements $a,b$ in $E$ are said to be \emph{orthogonal} (written $a\perp b$) if $L(a,b) =0$ (see \cite[Lemma\ 1]{BurgosFerGarMarPe2008} for additional details). Clearly this notion of orthogonality coincides with the usual one when it is restricted to the set of tripotents. A subset $\mathcal{S}\subseteq  E$ is called \emph{orthogonal} if $0\notin \mathcal{S}$ and $x\perp y$ for every $x\neq y$ in $\mathcal{S}$. The \emph{rank} of $E$ (denoted by rank$(E)$) is the minimal cardinal $r$ satisfying that $\sharp \mathcal{S}\leq r$ for every orthogonal subset of $E$. The rank of a tripotent $e$ in $E$ is defined as the rank of the JB$^*$-triple $E_2(e)$, and it is finite precisely when $e$ is a finite-rank tripotent. It is shown in \cite{BeLoPeRo2004} that a JB$^*$-triple has finite rank if and only if it is reflexive and that occurs if and only if it is a finite direct sum of finite rank Cartan factors. 

Every family $\{e_i\}_{i\in \Lambda}$ of mutually orthogonal tripotents in a JBW$^*$-triple $W$ is summable with respect to the weak$^*$-topology. Actually, the weak$^*$-limit of the sum $e= \hbox{w$^*$}-\sum_{i\in \Lambda} e_i\in W$ is another tripotent in  $W$ satisfying $e\geq e_i$ for all $i\in \Lambda$ (cf. \cite[Corollary 3.13]{horn1987ideal}). 

We state next the well known fact that the $U$-operator associated with a unitary element in a unital JB$^*$-algebra is a surjective linear isometry and hence a triple isomorphism.

\begin{lemma}\label{L:shift}{\rm (\cite[Proposition 4.3]{braun1978holomorphic} or \cite[Theorem 4.2.28$(vii)$]{Cabrera-Rodriguez-vol1})}
Let $B$ be a unital JB$^*$-algebra and let $u\in B$ be a unitary element. Then the mapping $T:B\to B$ defined by
$$T(x)=\J u{x^*}u = U_u (x),\quad x\in B$$
is a triple automorphism of $B$.
\end{lemma}

\begin{proof}
Note that
$$T(x)=(x^*)^{*_u},$$
hence $T$ is a linear isometric bijection of $B$. It follows from Kaup's theorem \cite[Proposition 5.5]{kaup1983riemann} that it is a triple automorphism.
\end{proof}

The following proposition also gathers some properties which are part of the folklore in JB$^*$-algebra theory. 

\begin{prop}\label{P:spectral}
Let $B$ be a finite-dimensional JB$^*$-algebra. Then the following assertions are valid.
\begin{enumerate}[$(a)$]
    \item  $B$ is unital.
    \item Any unitary element $u\in A$ may be expressed as
    $$u=\alpha_1 p_1+\dots+\alpha_n p_n,$$
    where $p_1,\dots,p_n$ are mutually orthogonal projections in $B$ with sum equal to $1$ and $\alpha_1,\dots,\alpha_n$ are distinct complex units.
    
    Moreover, this representation is unique up to reordering.
    \item Any self-adjoint element $x\in A$ may be expressed as
    $$x=\alpha_1 p_1+\dots+\alpha_n p_n,$$
    where $p_1,\dots,p_n$ are mutually orthogonal projections in $B$ with sum equal to $1$ and $\alpha_1,\dots,\alpha_n$ are distinct 
    real numbers.
    
    Moreover, this representation is unique up to reordering.
\end{enumerate}
\end{prop}

\begin{proof}
Assertion $(a)$ follows e.g. from Lemma 4.1.7 in  \cite{hanche1984jordan}.

$(b)$  Let $N$ be the Jordan $*$-subalgebra of $B$ generated by $u$. Then $N$ contains $u^*$ and $1$ and is associative. So, it is a finite-dimensional commutative C$^*$-algebra. Hence, it is $*$-isomorphic to $\ce^n$ for some $n\in\en$. In particular, $u$, being unitary, is of the form $(\alpha_1,\dots,\alpha_n)$, where the coordinates are complex units. This proves the existence of a representation of the required form.

Further, assume that $u$ is represented in this form. Let $M$ be the Jordan $*$-subalgebra of $B$ generated by the projections $p_1,\dots,p_n$. Then $M$ is a finite-dimensional commutative C$^*$-algebra containing $N$. Moreover, the spectrum of $x$ in $M$ is
$\sigma_M(x)=\{\alpha_1,\dots,\alpha_n\}$. It is finite, hence $\sigma_N(x)=\sigma_M(x)$. We conclude by the uniqueness of the spectral decomposition.

$(c)$ The proof is completely analogous to the proof of $(b)$, except that $\alpha_1,\dots,\alpha_n$ are real numbers.
\end{proof}

\begin{lemma}\label{L:sqrt}
Let $B$ be a finite-dimensional JB$^*$-algebra and let $u\in B$ be a unitary element. Then there is a unitary element $v\in B$ such that $v^2=u$.
\end{lemma}

\begin{proof}
It is enough to consider the representation from Proposition~\ref{P:spectral}$(b)$ and take some square roots of the coefficients $\alpha_1,\dots,\alpha_n$.
\end{proof}

\section{Construction and basic properties of the Cartan factor of type 6}\label{sec:CD-C6}

In this section we recall the definition of the Cartan factor of type $6$ and its basic properties. Usually it is represented as the JB$^*$-algebra of hermitian $3\times 3$ matrices of complex octonions. In order to present its structure we first need to recall the Cayley-Dickson doubling process (see, e.g., \cite[\S6.1.30]{Cabrera-Rodriguez-vol2} for an abstract approach). We will restrict ourselves to a special case and we will use the notation from \cite{Finite}.

\subsection{Cayley-Dicskon doubling process starting from $\ce$ or $\er$}\label{subsec: Cayley-Dickson process}

By induction we define (in general non-associative) complex algebras $\A_n$ for $n\ge0$ equipped with a product $\boxdot_n$, conjugation $\overline{\  \cdot \ }$ and two involutions -- a linear one denoted by $^{\inv_n}$ and a conjugate-linear one denoted by $^{*_n}$.

$\A_0$ is the complex field, i.e., $\A_0=\ce$, $\boxdot_0$ is the standard multiplication of complex numbers, $^{\inv_0}$ is the identity and $^{*_0}$ is the complex conjugation. 

Given the structure defined on $\A_n$ we define the corresponding product, conjugation and involutions on $\A_{n+1}=\A_n\times \A_n$ as follows:
$$\begin{aligned}
\overline{(x_1,x_2)}&=(\overline{x_1},\overline{x_2}),\\
(x_1,x_2)^{\inv_{n+1}}&=(x_1^{\inv_n},-x_2),\\
(x_1,x_2)^{*_{n+1}}&=(x_1^{*_n},-\overline{x_2},)\\
(x_1,x_2)\boxdot_{n+1}(y_1,y_2)&=(x_1\boxdot_n y_1-y_2\boxdot_n x_2^{\inv_n},x_1^{\inv_n} \boxdot_n y_2+y_1\boxdot_n x_2)
\end{aligned}$$
for $(x_1,x_2), (y_1,y_2)\in \A_n\times \A_n=\A_{n+1}$.

It is well known that $\A_n$ is a (possibly) non-associative complex algebra (i.e., the mapping $(x,y)\mapsto x\boxdot_n y$ is bilinear) of dimension $2^n$, $^{*_n}$ is a conjugate linear involution and $^{\inv_n}$ is a linear involution. Moreover, $\A_n$ is a subalgebra of $\A_{n+1}$ (if $a\in \A_n$ is identified with $(a,0)\in\A_{n+1}$. These basic properties together with some more are summarized in \cite[Lemma 6.5]{Finite}.
 
We observe that the two involutions are related with the conjugation in such a way that for any $x\in\A_n$ we have
$$x^{*_n}=\overline{x}^{\inv_n}=\overline{x^{\inv_n}},\quad x^{\inv_n}=\overline{x}^{*_n}=\overline{x^{*_n}},\quad \overline{x}=(x^{\inv_n})^{*_n}=(x^{*_n})^{\inv_n}.$$
 
There is also a real variant of this process which produces (in general non-associative) real algebras $(\A_n)_R$ for $n\ge0$ equipped with a product $\boxdot_n$ and a linear involution $^{\inv_n}$. There are two equivalent ways to get them -- either we set $(\A_0)_R=\er$, the product $\boxdot_0$ is then the multiplication of real numbers and $^{\inv_0}$ is the identity and we define inductively $(\A_{n+1})_R=(\A_n)_R\times(\A_n)_R$ with the operations $\boxdot_{n+1}$ and $^{\inv_{n+1}}$ defined by the above formula; or we take $(\A_n)_R$ to be $\er^{2^n}$ consider as the canonical real-linear subspace of $\A_n$ with inherited operations, i.e.,
$$(\A_n)_R=\{x\in\A_n\setsep \overline{x}=x\}.$$
In this way $(\A_n)_R$ is a real subalgebra of $\A_{n}$. To simplify the notation, in the sequel we will omit the index $n$ at the two involutions on $\A_n$.

There are two natural norms on $\A_n$. The first one is the Hilbertian norm coming from the standard inner product (note that $\A_n$ may be canonically identified with $\ce^{2^n}$). The inner product will be denoted by $\ip{\cdot}{\cdot}$ and the resulting norm by $\norm{\cdot}_2$. 

The second one is the norm of the spin factor defined by the formula
$$\norm{x}^2=\norm{x}_2^2+\sqrt{\norm{x}_2^4-\abs{\ip{x}{\overline{x}}}^2}.$$
Note that on $(\A_n)_R$ the two norms coincide.

If $\A_n$ is equipped with the above-defined norm  $\norm{\cdot}$ and the triple product
$$\J xyz=\ip{x}{y}z+\ip{z}{y}x-\ip{x}{\overline{z}}\overline{y},$$
it becomes a JB$^*$-triple (called a spin factor, cf. Section~\ref{sec:2}). Moreover, the algebra $\A_n$ has a unit $1\in\A_0\subset\A_n$. It is a unitary element of the above-defined JB$^*$-triple, so it produces a structure of a unital JB$^*$-algebra on $\A_n$. The interplay of these structures is summarized in the following lemma which follows from \cite[Lemma 6.6]{Finite}.

\begin{lemma}\label{L:CD} Let $n$ be a non-negative integer.
\begin{enumerate}[$(a)$]
    \item  $\ip{x}{y}=\ip{x^\inv}{y^\inv}=\frac12(x\boxdot_n y^*+\overline{y}\boxdot_n x^\inv)$ for $x,y\in\A_n$.
    \item $x^*=\J 1x1$ for $x\in \A_n$, hence the involution $^*$ coincides with the involution on $\A_n$ from its structure of JB$^*$-algebra.
    \item $x\circ y=\J x1y=\frac12(x\boxdot_n y+y\boxdot_n x)$ for $x,y\in \A_n$.
    \item $\ip{x}{y}=\ip{x^\inv}{y^\inv}=\frac12(x\circ y^*+\overline{y}\circ x)$ for $x,y\in\A_n$.
\end{enumerate}
\end{lemma}

\subsection{The first two steps -- towards quaternions and biquaternions}\label{subsec: quaternions and biquaternions}
 Although algebras $\A_n$ may be defined for any $n\ge0$, the most important ones are those for $n=0,1,2,3$.  We start by recalling their basic properties for $n\le2$.

\begin{enumerate}[$(1)$]
    \item By the very definition $\A_0$ is the commutative field of complex numbers and $(\A_0)_R$ is the commutative field of real numbers.

\item $(\A_1)_R$ is canonically isomorphic the the complex field -- the isomorphism is $(x_1,x_2)\mapsto x_1+i x_2$. The involution $^\inv$ then corresponds to the complex conjugation.

So, both $\A_0$ and $(\A_1)_R$ are isomorphic to the complex field, but $\A_0$ as a complex algebra with the conjugate-linear involution given by the complex conjugation and $(\A_1)_R$ as a real algebra with the linear involution given by the complex conjugation.

\item (cf. \cite[Lemma 6.7$(iv)$]{Finite}) $\A_1$ is a commutative C$^*$-algebra $*$-isomorphic to $\ce\oplus\ce$. The witnessing isomorphism is
$(x_1,x_2)\mapsto (x_1+i x_2,x_1-i x_2):\A_1\to  \ce\oplus\ce $. In this representation  the conjugation and involutions on $\A_1$ are then defined by
$$\begin{aligned}
\overline{(x_1,x_2)}&=(\overline{x_2},\overline{x_1})\\
(x_1,x_2)^\inv&=(x_2,x_1),\\ (x_1,x_2)^*&=(\overline{x_1},\overline{x_2})
\end{aligned}$$
for $(x_1,x_2)\in\ce\oplus_\infty\ce$.

Moreover, we have the following identifications:
$$\begin{aligned}
\A_0&=\{(x,x)\setsep x\in\ce\},\\
(\A_0)_R&=\{(x,x)\setsep x\in\er\},\\
(\A_1)_R&=\{(x,\overline{x})\setsep x\in\ce\}.
\end{aligned}$$

\item $(\A_2)_R$ is canonically isomorphic to the non-commutative fields of quaternions. It is usually denoted by $\Ha$.

\item (cf. \cite[Lemma 6.7$(i),(ii)$]{Finite}) $\A_2$ is canonically isomorphic to the algebra of biquaternions -- quaternions with complex coefficients. We will denote it by $\Ha_C$. It is a non-commutative C$^*$-algebra which is $*$-isomorphic to $M_2$, the algebra of $2\times 2$ complex matrices. A witnessing isomorphism is given by
$$(x_1,x_2,x_3,x_4)\mapsto \begin{pmatrix}
x_1+i x_2 & x_3+i x_4 \\ -x_3+i x_4 & x_1-i x_2
\end{pmatrix}.$$
In this identification the conjugation and the two involutions on $\Ha_C$ are described by
$$\begin{aligned}
\overline{\begin{pmatrix}a&b\\c& d\end{pmatrix}}&=\begin{pmatrix}\overline{d}&-\overline{c}\\-\overline{b}&\overline{a}
\end{pmatrix},\\
\begin{pmatrix}a&b\\c& d\end{pmatrix}^\inv&=\begin{pmatrix}{d}&-{b}\\-{c}&{a}
\end{pmatrix},\\
\begin{pmatrix}a&b\\c& d\end{pmatrix}^*&=\begin{pmatrix}\overline{a}&\overline{c}\\\overline{b}&\overline{d}
\end{pmatrix}.
\end{aligned}$$
Moreover, in this identification we have:
\begin{enumerate}[$(i)$]
    \item $\A_1$ corresponds to the subalgebra of $M_2$ formed by diagonal matrices.
    \item $(\A_1)_R$ corresponds to the matrices of the form $\begin{pmatrix}
    a&0\\0&\overline{a}
    \end{pmatrix}$, $a\in\ce$.
    \item $\A_0$ corresponds to complex multiples of the unit matrix and $(\A_0)_R$ to real multiples.
    \item $(\A_2)_R$ corresponds to the matrices of the form 
    $\begin{pmatrix}
    a&b\\-\overline{b}&\overline{a}
    \end{pmatrix}$, $a,b\in\ce$.
\end{enumerate}
\end{enumerate}

\begin{remark}
\begin{enumerate}[$(a)$]
    \item It follows from the above identifications that the products $\boxdot_n$ are associative for $n\le 2$ and commutative for $n\le 1$. In the sequel we will denote these associative products by $\cdot$ or simply omit them as it is usual.
    
    \item Note that in the above representations of $\A_1$ and $\A_2$ the conjugation does not coincide with the coordinatewise conjugation on $\ce\oplus_\infty\ce$ or the entrywise conjugation on $M_2$. 
    The reason is that the conjugation depends on the choice of a basis. 
    
    In particular, the conjugation on $\ce\oplus_\infty\ce$ coming from $\A_1$ is the coordinatewise conjugation with respect to the basis $(1,1),(i,-i)$ and the conjugation on $M_2$ coming from $\A_2$ is the coordinatewise conjugation with respect to the basis
    $$\begin{pmatrix}
    1&0\\0&1
    \end{pmatrix},\begin{pmatrix}
    i&0\\0&-i
    \end{pmatrix},\begin{pmatrix}
    0&1\\-1&0
    \end{pmatrix},\begin{pmatrix}
    0&i\\i&0
    \end{pmatrix}.$$
\end{enumerate}
\end{remark}

\subsection{Real and complex octonions}\label{subsect: real and complex octonions}

While the algebras $\A_n$ for $n\le2$ are associative, for $n=3$ the situation is different. The algebra $(\A_3)_R$ is known under the name \emph{Cayley numbers} or \emph{(real) octonions} and we will denote it by $\O_R$.
Further, the algebra $\A_3$ is known as \emph{complex Cayley numbers} or \emph{(complex) octonions} and we will denote it by $\O$ or $\O_C$. Further, the product $\boxdot_3$ will be denoted simply by $\boxdot$ in the sequel.

Let us now recall some properties of the algebra $\O$.

\begin{prop}\label{P:octonions} \
\begin{enumerate}[$(i)$]
    \item The product $\boxdot$ on $\O$ is neither commutative nor associative, but it is alternative, i.e.,
    $$x\boxdot(x\boxdot y)=(x\boxdot x)\boxdot y\mbox{ and }y\boxdot(x\boxdot x)=(y\boxdot x)\boxdot x\mbox{ for  }x,y\in\O.$$
    \item $\ip{x\boxdot z^*}{y}=\ip{x}{y\boxdot z}$ and $\ip{z^*\boxdot x}{y}=\ip{x}{z\boxdot y}$ for $x,y,z\in\O$.
    \item $\J xyz=\frac12(x\boxdot(y^*\boxdot z)+z\boxdot(y^*\boxdot x))=\frac12((x\boxdot y^*)\boxdot z+(z\boxdot y^*)\boxdot x)$ for $x,y,z\in \O$.
\end{enumerate}
\end{prop}

\begin{proof}
Assertion $(i)$ is well known, see e.g.,  \cite[Lemma 6.5$(xii)$]{Finite}.  Assertions $(ii)$ and $(iii)$ are proved for example in \cite[Lemma 6.6$(d),(e)$]{Finite}.
\end{proof}

We continue by some properties of real octonions.

\begin{prop}\label{P:real octonions}
\begin{enumerate}[$(i)$]
    \item $x\boxdot x^\inv=x^\inv\boxdot x=\norm{x}^2$ for $x\in\O_R$.
    \item $x\boxdot (x^\inv \boxdot y)=(y\boxdot x^\inv)\boxdot x=\norm{x}^2 y$ for $x,y\in\O_R$.
    \item If $x,y,z\in\O_R$ and $x\ne0$, then
    $$\begin{aligned}
    y\boxdot x=z\boxdot x&\Longrightarrow y=z,\\
    x\boxdot y=x\boxdot z&\Longrightarrow y=z.
    \end{aligned}$$
    \item $\O_R$ is a division algebra, i.e., whenever $x,y\in\O_R$ and $y\ne0$, then there are unique elements $u,v\in \O_R$ such that $x=u\boxdot y=y\boxdot v$.
\end{enumerate}
\end{prop}

\begin{proof}
$(i)$ For $x\in\O_R$ we have $x=\overline{x}$ and $x^*=x^\inv$, hence the assertion follows from Lemma~\ref{L:CD}$(a)$. (Note that this is not specific for $\O_R$, the analogue holds on $(\A_n)_R$ for each $n\ge0$.)

$(ii)$ If $x=0$, the equalities are obvious. Assume $x\ne0$. Then $\frac{x}{\norm{x}}$ is a unitary element of $\O$ (see e.g. \cite[Lemma 6.1$(a)$]{Finite}). Using Proposition~\ref{P:octonions}$(iii)$ we get
$$x\boxdot (x^\inv\boxdot y)=2\J xxy-(x\boxdot x^\inv)\boxdot y=2\norm{x}^2y-\norm{x}^2\boxdot y=\norm{x}^2y,$$
where we also used assertion $(i)$. The remaining equality is completely analogous.

$(iii)$ This follows easily from $(ii)$ -- it is enough to multiply the equality by $x^\inv$ from the right in the first case and from the left in the second case.

$(iv)$ The uniqueness follows immediately from $(iii)$. Further, by $(ii)$ we may take
$u=\frac{1}{\norm{y}^2}x\boxdot y^\inv$ and $v=\frac{1}{\norm{y}^2}y^\inv\boxdot x$.
\end{proof}

\subsection{Matrices of octonions and the Cartan factor of type 6}\label{subsec: H3O}

We may consider matrices with entries in $\O$. The two involutions $^*$ and $^\inv$ for such matrices and the product of two matrices of compatible types (denoted again by $\boxdot$) are defined in the standard way: 
If $A=(a_{ij})$ is a matrix of type $m\times n$, then $A^*$ and $A^\inv$ are matrices of type $n\times m$; on the place $ij$ the first one has the element $a_{ji}^*$ and the second one the element $a_{ji}^\inv$. Moreover, if $A=(a_{ij})$ is a matrix of type $m\times n$ and $B=(b_{jk})$ is a matrix of type $n\times p$, then 
$A\boxdot B$ is the matrix of type $m\times p$ which has the element $\sum_{j=1}^n a_{ij}\boxdot b_{jk}$ on place $ik$.  

The Cartan factor of type $6$ is the JB$^*$-algebra of $^\inv$-hermitian $3\times 3$ matrices of octonions, i.e.,
$$C_6=H_3(\O)=\{\x\in M_3(\O)\setsep \x^\inv=\x\},$$
equipped with the Jordan product
$$\x\circ \y=\frac12(\x\boxdot \y+\y\boxdot\x),$$
the involution $^*$ and a uniquely determined norm (cf. \cite{Wright1977} or \cite[\S6.1.38]{Cabrera-Rodriguez-vol2}).
Moreover, as any JB$^*$-algebra, $C_6$ becomes a JB$^*$ triple under the triple product given in \eqref{eq triple product JBstar algebras}.
Note that for $\x\in C_6$ we have $\x^*=\overline{\x}$, where the conjugation is considered entrywise. A general element of $C_6$ has the form
$$\x=\begin{pmatrix}
\alpha & a & b \\ a^\inv & \beta & c \\ b^\inv & c^\inv & \gamma
\end{pmatrix},$$
where $\alpha,\beta,\gamma\in\ce$ and $a,b,c\in\O$. Since it is determined by three complex numbers and three octonions, clearly $\dim C_6=27$. Further, a general element $\x$ is self-adjoint, i.e., $\x^*=\x$, if and only if $\alpha,\beta,\gamma\in\er$ and $a,b,c\in\O_R$. Self-adjoint elements form the exceptional JB-algebra $H_3(\O_R)$.
 
The JB$^*$-algebra $C_6$ is unital, its unit is the unit matrix, we will denote it by $\1$. Note that we even have 
$$\1\boxdot\x=\x\boxdot\1\mbox{ for each }\x\in M_3(\O).$$
Further, $C_6$ is of rank three (cf. \cite[Table 1]{kaup1997real})-- the unit $\1$ is the sum of three mutually orthogonal minimal projections
$$\1=\begin{pmatrix}
1&0&0\\0&0&0\\0&0&0
\end{pmatrix}+\begin{pmatrix}
0&0&0\\0&1&0\\0&0&0
\end{pmatrix}+\begin{pmatrix}
0&0&0\\0&0&0\\0&0&1
\end{pmatrix}$$
and, more generally, any unitary element is the sum of three mutually orthogonal minimal tripotents.

A \emph{frame} in a Cartan factor $C$ is an orthogonal family $\{e_i\}_{i\in \Lambda}$ of minimal tripotents in $C$ for which the tripotent $e= \hbox{w}^*\hbox{-}\sum_{i\in \Lambda} e_i$ is complete and satisfies that dim$(C_1(e))\leq$ dim$(C_1(\tilde{e}))$ for any other complete tipotent $\tilde{e}\in C$.

The situation in $C_6$ is easier. Since $C_6$ is a finite-dimensional JB$^*$-algebra, any complete tripotent is unitary by \cite[Propositions 3.3 and 3.4]{Finite}. Moreover, any two unitary elements of $C_6$ may be exchanged by a triple automorphism (this follows easily from Lemma~\ref{L:sqrt} and Lemma~\ref{L:shift}). If we combine it with the fact that $C_6$ has rank three, we deduce that any triple of mutually orthogonal minimal tripotents in $C_6$ is a frame.

\begin{lemma}\label{L:frame-automorphism}
Let $\uu_1,\uu_2,\uu_3$ and $\vv_1,\vv_2,\vv_3$ be two triples of mutually orthogonal minimal tripotents in $C_6$. Then the following assertions hold.
\begin{enumerate}[$(a)$]
    \item $\uu_1+\uu_2+\uu_3$ and $\vv_1+\vv_2+\vv_3$ are unitary elements in $C_6$.
    \item There is a triple automorphism $T:C_6\to C_6$ such that $T(\uu_j)=\vv_j$ for $j=1,2,3$.
    \item If $\uu_1,\uu_2,\uu_3,\vv_1,\vv_3,\vv_3$ are projections, then any of the triple automorphisms provided by $(b)$ are Jordan $*$-automorphisms.
\end{enumerate}
\end{lemma}

\begin{proof} As we have seen above, the unit element $\1$ in $C_6$ writes as the orthogonal sum of three mutually orthogonal minimal tripotents $\e_1,\e_2,\e_3$. By  \cite[Proposition 5.8]{kaup1997real}$(i)$ and $(ii)$ each finite orthogonal family of minimal tripotents in a Cartan factor $C$ can be extended to a frame, and the cardinality of every frame in $C$ coincides with the rank of $C$.  Furthermore, by \cite[Proposition 5.8]{kaup1997real}$(iii)$ any two frames in $C$ can be interchanged by a triple automorphism on $C$. So, there exists a triple automorphism $T$ on $C_6$ satisfying $T_1 (\e_j) = \uu_j$ for all $j=1,2,3$. Since $\1$ is a unitary in $C_6,$ the same property passes to $T(\1)= T(\e_1) + T(\e_2) +T(\e_3) = \uu_1+\uu_2+\uu_3 .$ This proves $(a)$. 

Similar arguments from those given above can be applied to derive $(b)$ from \cite[Proposition 5.8]{kaup1997real}$(i),$ $(ii)$ and $(iii)$ (furthermore, since we are working with a finite dimensional Cartan factor the results in \cite{kaup1997real} can be replaced by \cite[Proposition 5.2 and Theorem 5.3]{loos1977bounded}).

$(c)$ If $\uu_1,\uu_2,\uu_3,\vv_1,\vv_3,\vv_3$ are projections, the elements $\uu =\uu_1+\uu_2+\uu_3$ and $\vv=\vv_1+\vv_3+\vv_3$ are projections too. Since they clearly have rank 3 and are bounded by $\1,$ they both coincide with $\1$. Then any of the triple automorphisms given by $(b)$ is unital and hence a Jordan $^*$-isomorphism (cf. \eqref{eq unital triple homomorphisms are algebra homomorphisms}).
\end{proof}

\section{Unitaries in $C_6$}\label{sec:unitaries}

In this section we present several results on the structure of unitary elements in the Cartan factor of type $6$. They were the original motivation of our research and will be used in a forthcoming paper. 

The proofs of some of the results will be done later using results on projections, automorphisms and matrices of biquaternions in the following sections. 

We start by the following theorem on spectral decomposition of unitary elements. It is not really new as it easily follows from known results, but it serves as a starting point for further results.

\begin{thm}\label{T:spectral decomposition of unitary}
  Let $\uu\in C_6$ be a unitary element. 
  \begin{enumerate}[$(i)$]
      \item There are complex units $\alpha_1,\alpha_2,\alpha_3$ and mutually orthogonal minimal projections $\p_1,\p_2,\p_3$ 
such that $u=\alpha_1\p_1+\alpha_2\p_2+\alpha_3\p_3$.
\item The representation from $(i)$ is unique in the natural sense: The triple $(\alpha_1,\alpha_2,\alpha_3)$ is uniquely determined up to reordering. Further, for any complex unit $\alpha$ the sum $\sum_{\alpha_j=\alpha}\p_j$ is also uniquely determined.
\item There is a Jordan $*$-automorphism $T$ of $C_6$ such that $T(\uu)$ is a diagonal matrix.
  \end{enumerate}
\end{thm}

\begin{proof}
$(i)$  
By Proposition~\ref{P:spectral}$(b)$ we have $\uu=\sum_{j=1}^n\alpha_j \p_j$, where $\alpha_1,\dots,\alpha_n$ are complex units and  $\p_1,\dots,\p_n$ are mutually orthogonal projections with sum equal to $\1$.
Since the rank of $C_6$ is three, necessarily $n\le 3$. If all $\p_j$ are minimal,
then $n=3$. If some of the projections $\p_j$ has higher rank,
it may be decomposed as the sum of  minimal projections. This completes the proof of $(i)$.

$(ii)$ The uniqueness follows easily from Proposition~\ref{P:spectral}$(b)$.

$(iii)$ This follows from $(i)$ using Lemma~\ref{L:frame-automorphism}.
\end{proof}

The previous theorem says, in particular, that to each unitary element $\uu\in C_6$ we may canonically assign a unique triple of complex units $\alpha_1,\alpha_2,\alpha_3$ (it is unique up to reordering and these three numbers need not be distinct).
We define the \emph{determinant} of such $\uu$ by the formula
$$\dt \uu=\alpha_1\alpha_2\alpha_3.$$
It is clear that the triple $\alpha_1,\alpha_2,\alpha_3$ (and hence also the determinant) is preserved by Jordan $*$-automorphisms of $C_6$.

Further, if $\e\in C_6$ is a unitary element, we may introduce on $C_6$ a new structure of a JB$^*$-algebra in which $\e$ is the unit. The operations are defined by
$$\x\circ_{\e}\y=\J{\x}{\e}{\y}\mbox{ and }\x^{*_{\e}}=\J {\e}{\x}{\e}$$
for $\x,\y\in C_6$. This JB$^*$-algebra is Jordan $*$-isomorphic to $C_6$. One of the ways  to prove this is to observe that $\e$ may be expressed as the sum of three mutually orthogonal minimal tripotents and use Lemma~\ref{L:frame-automorphism}. 

Therefore, we may apply Theorem~\ref{T:spectral decomposition of unitary} to this new JB$^*$-algebra and deduce that, given $\uu\in C_6$ unitary, there is a decomposition
$$\uu=\alpha_1\vv_1+\alpha_2\vv_2+\alpha_3\vv_3,$$
where $\vv_1,\vv_2,\vv_3$ are mutually orthogonal minimal tripotents satisfying $\vv_j\le \e$ for $j=1,2,3$. Moreover, this decomposition is unique in the sense of Theorem~\ref{T:spectral decomposition of unitary}$(ii)$. So, in this situation we may define
$$\dt_{\e}\uu=\alpha_1\alpha_2\alpha_3.$$

The key structure result is the following theorem which will be proved in Section~\ref{sec:proofs} below.

\begin{thm}\label{T:product}
Let $\uu,\e\in C_6$ be two unitary elements. Then
$$\dt\uu=\dt_{\e}\uu\cdot\dt\e.$$
\end{thm}

Next we collect two corollaries of this theorem.

\begin{cor}
Let $\uu,\e\in C_6$ be two unitary elements. Assume that 
$\uu$ is self-adjoint in the JB$^*$-algebra $C_6$ with unit $\e$ (i.e., $\uu=\J{\e}{\uu}{\e}$). Then
$$\dt\uu=\dt\e\mbox{ or }\dt\uu=-\dt\e.$$
\end{cor}

\begin{proof}
A self-adjoint tripotent in a JB$^*$-algebra is the difference of two mutually orthogonal projections. Hence, under our assumptions $\uu=\uu_1-\uu_2$, where $\uu_j\le\e$ for $j=1,2$, and $\uu_1+\uu_2=\e$. It follows that the numbers $\alpha_j$ from the above-described decomposition of $\uu$ with respect to $\e$ are $1$ or $-1$. Thus $\dt_{\e}\uu=1$ or $\dt_{\e}\uu=-1$. Now we may conclude using Theorem~\ref{T:product}.
\end{proof}

\begin{cor}\label{cor:dt after triple-aut}
Let $\uu\in C_6$ be a unitary element and let $T:C_6\to C_6$ be a triple automorphism. Then
$$\dt T(\uu)=\dt \uu\cdot\dt T(\1).
$$
\end{cor}

\begin{proof}
Let 
$$\uu=\alpha_1\p_1+\alpha_2\p_2+\alpha_3\p_3$$
be the decomposition from Theorem~\ref{T:spectral decomposition of unitary}. Then
$$T(\uu)=\alpha_1T(\p_1)+\alpha_2T(\p_2)+\alpha_3T(\p_3)$$
and $T(\p_j)\le T(\1)$ for $j=1,2,3$. It follows that
$$\dt_{T(\1)}T(\uu)=\dt \uu,$$
hence we may conclude by Theorem~\ref{T:product}.
\end{proof}

An important tool to prove Theorem~\ref{T:product} is the following theorem which is also interesting in itself. It will be also proved in Section~\ref{sec:proofs} below.

\begin{thm}\label{T:simultaneous biq}
Let $\uu,\e\in C_6$ be two unitary elements. Then there is a Jordan $*$-automorphism $T:C_6\to C_6$ such that
\begin{enumerate}[$(i)$]
    \item $T(\e)$ is a diagonal matrix.
    \item The entries of $T(\uu)$ are biquaternions.
\end{enumerate}
\end{thm}

Let us comment a bit the meaning of this theorem. It says, in particular, that given two unitary elements in $C_6$, we may assume, up to applying a Jordan $*$-automorphism, that they belong to $H_3(\Ha_C)$, a Jordan $*$-subalgebra of the C$^*$-algebra $M_3(\Ha_C)$ (which is $*$-isomorphic to $M_6$). 

So, not only that the Jordan $*$-subalgebra of $C_6$ generated by two unitary elements is a JC$^*$-algebra (which is known -- it easily follows using the functional calculus and \cite[Corollary 2.2 and subsequent comments]{Wright1977}), but the surrounding C$^*$-algebra may be $M_3(\Ha_C)$ and the injection may be induced by a Jordan $*$-automorphism of $C_6$, preserving hence all the structure, including the determinants.

\section{Hermitian matrices of biquaternions and their determinants}\label{sec:dt-n}

Recall that biquaternions (denoted by $\Ha_C$) are a non-commutative C$^*$-algebra $*$-isomorphic to the matrix algebra $M_2$. Therefore $M_n(\Ha_C)$ -- $n\times n$ matrices of biquaternions -- is a C$^*$-algebra $*$-isomorphic to $M_{2n}$. We will consider the subspace formed by $^\inv$-hermitian matrices, i.e., 
$$H_n(\Ha_C)=\{ \x\in M_n(\Ha_C)\setsep \x^\inv=\x\}.$$
This subspace is a Jordan $*$-subalgebra of $M_n(\Ha_C)$. We will define and investigate determinants of elements in $H_n(\Ha_C)$. 

Note that in Section~\ref{sec:unitaries} we defined determinants of unitary elements in $C_6=H_3(\O)$ using their spectral decomposition. Now we are going to define determinants of general elements of $H_n(\Ha_C)$ using an inductive formula inspired by the rules of computing determinants of complex matrices. In Section~\ref{sec:proofs} below we will show that for unitary elements in $H_3(\Ha_C)$ the two approaches give the same result.

For our original motivation it would be enough to work only with $3\times 3$ matrices, but we find interesting that the theory works for a general $n$.

We will further need some notation reflecting the correspondence between $\Ha_C$ and $M_2$. If $x\in \Ha_c$, we will denote by $\widehat{x}$ the corresponding $2\times 2$ complex matrix. Further, if $\x\in M_n(\Ha_C)$, by $\widehat{\x}$ we will denote the corresponding $2n\times 2n$ complex matrix (cf. $(5)$ in subsection \ref{subsec: quaternions and biquaternions}).

The promised determinants are introduced by the following theorem, where we also gather their basic properties.

\begin{thm}\label{T:determinant}
There is a unique sequence of mappings $(\dt_n)_{n\in\en}$ with the following properties.
\begin{enumerate}[$(i)$]
    \item $\dt_n:H_n(\Ha_C)\to\Ha_C$ is a continuous mapping for each $n\in\en$.
    \item If $\x=(x_{11})\in H_1(\Ha_C)$, then $\dt_1 \x=x_{11}$.
    \item If $\x=(x_{ij})\in H_{n+1}(\Ha_C)$ and $x_{11}\ne0$, then
    $$\dt_{n+1}\x=x_{11}\cdot\dt_n \left(x_{ij}-x_{11}^{-1}x_{i1}x_{1j}\right)_{2\le i,j\le n+1}.$$
    \end{enumerate}    

Moreover, the following assertions hold as well.
    
\begin{enumerate}[$(i)$]\setcounter{enumi}{3}     
\item The values of $\dt_n$ are complex numbers.
\item $\dt_n (\alpha\x)=\alpha^n\dt_n\x$ for $\x\in H_n(\Ha_C)$ and $\alpha\in\ce$.
    \item $\dt_n\x=\sum_{\sigma,\pi\in\mathbb{S}_n}\alpha_{\sigma,\pi}\cdot x_{\sigma(1),\pi(1)}\cdots x_{\sigma(n),\pi(n)}$ for some complex coefficients $\alpha_{\sigma,\pi}$. (Here $\mathbb{S}_n$ denotes the set of all permutations of $\{1,\dots,n\}$.)
   
    \item $\det \widehat{\x}=\left(\dt_n\x\right)^2$ for $\x\in H_n(\Ha_C)$, here $\det {\widehat{x}}$ stands for the usual determinant of matrices $\widehat{x}\in M_{2n} (\mathbb{C})$.
    \item If $\x,\y\in H_n(\Ha_C)$, then the function $\lambda\mapsto\dt_n(\lambda\y+\x)$ is a complex polynomial of degree at most $n$ with coeffients $\dt_n\y$ at $\lambda^n$ and $\dt_n\x$ at $\lambda^0$.
    \item If $\x\in H_n(\Ha_C)$, then each eigenvalue of $\widehat{\x}$ has even multiplicity. Moreover, $\dt_n\x$ is the product of all eigenvalues of $\widehat{\x}$, each one counted with half a multiplicity.
\end{enumerate}
\end{thm}

\begin{proof} 
Let us start by explaining that the formula from $(iii)$ is reasonable. Assume that $\x=(x_{ij})\in H_{n+1}(\Ha_C)$ and $x_{11}\ne0$. Since $\x=\x^\inv$, necessarily $x_{11}\in\ce$, thus the inverse $x_{11}^{-1}$ exists. Moreover, the matrix on the right-hand side is $^\inv$-hermitian, as
$$(x_{ij}-x_{11}^{-1}x_{i1}x_{1j})^\inv=x_{ij}^\inv-x_{11}^{-1}x_{1j}^\inv x_{i1}^\inv=x_{ji}-x_{11}^{-1}x_{j1}x_{1i}.$$
Thus the formula from $(iii)$ is an inductive formula defining $\dt_{n+1}\x$ in case $x_{11}\ne0$ using $\dt_n$.

We will prove by induction the existence of a sequence $(\dt_n)$ satisfying conditions $(i)-(viii)$. The uniqueness is then easy -- $\dt_1$ must be defined as in $(ii)$. Moreover, condition $(iii)$ determines uniquely $\dt_{n+1}\x$ in case $x_{11}\ne0$. Such matrices are dense in $H_{n+1}(\Ha_C)$, so the uniqueness follows by continuity. Assertion $(ix)$ will be deduced from $(i)-(viii)$ at the end of the proof.

Let us start by $n=1$. The mapping $\dt_1$ is defined in $(ii)$ and it is clearly continuous, so $(i)$ holds. There is nothing to check in assertion $(iii)$ as $\dt_2$ is not defined, and assertions $(iv)-(vi)$ and $(viii)$  are in this case obvious. Further, if $\x=(x_{11})$, then $\widehat{\x}=\begin{pmatrix}
x_{11}&0\\0&x_{11}
\end{pmatrix},$
hence assertions $(vii)$  follows easily.

Next assume that $n\in\en$ and we have mappings $\dt_1,\dots,\dt_n$ satisfying conditions $(i)-(viii)$. We will show how $\dt_{n+1}$
may be defined and prove the respective properties.  

We  start by defining $\dt_{n+1}\x$ for $\x\in H_{n+1}(\Ha_C)$ such that $x_{11}\ne0$  using the formula in $(iii)$. This surely may be done. 
It follows from the induction hypothesis that $\dt_{n+1}\x\in\ce$ whenever $x_{11}\ne0$. 
Further,
$$\begin{aligned}
\det\widehat{\x}&=\det\left(\begin{smallmatrix}
\widehat{x_{11}}&\widehat{x_{12}}&\dots&\widehat{x_{1,n+1}}\\\widehat{x_{21}}&\widehat{x_{22}}&\dots&\widehat{x_{2,n+1}}\\\vdots&\vdots&\ddots&\vdots\\
\widehat{x_{n+1,1}}&\widehat{x_{n+1,2}}&\dots&\widehat{x_{n+1,n+1}}
\end{smallmatrix}\right) \\&=
\det\left(\begin{smallmatrix}
\widehat{x_{11}}&\widehat{x_{12}}&\dots&\widehat{x_{1,n+1}}\\ 0&\widehat{x_{22}}-\widehat{x_{21}}\cdot\widehat{x_{11}}^{-1}\cdot\widehat{x_{12}}&\dots&\widehat{x_{2,n+1}}-\widehat{x_{21}}\cdot\widehat{x_{11}}^{-1}\cdot\widehat{x_{1,n+1}}\\\vdots&\vdots&\ddots&\vdots\\
0&\widehat{x_{n+1,2}}-\widehat{x_{n+1,1}}\cdot\widehat{x_{11}}^{-1}\cdot\widehat{x_{12}}&\dots&\widehat{x_{n+1,n+1}}-\widehat{x_{n+1,1}}\cdot\widehat{x_{11}}^{-1}\cdot\widehat{x_{1,n+1}}\end{smallmatrix}\right)
\\&=\det\widehat{x_{11}}\cdot
\det\left(\begin{smallmatrix}
\widehat{x_{22}}-\widehat{x_{21}}\cdot\widehat{x_{11}}^{-1}\cdot\widehat{x_{12}}&\dots&\widehat{x_{2,n+1}}-\widehat{x_{21}}\cdot\widehat{x_{11}}^{-1}\cdot\widehat{x_{1,n+1}}\\\vdots&\ddots&\vdots\\
\widehat{x_{n+1,2}}-\widehat{x_{n+1,1}}\cdot\widehat{x_{11}}^{-1}\cdot\widehat{x_{12}}&\dots&\widehat{x_{n+1,n+1}}-\widehat{x_{n+1,1}}\cdot\widehat{x_{11}}^{-1}\cdot\widehat{x_{1,n+1}}\end{smallmatrix}\right)
\\&=x_{11}^2\cdot \left(\dt_n \left(x_{ij}-x_{11}^{-1}x_{i1}x_{1j}\right)_{2\le i,j\le n+1}\right)^2=\left(\dt_{n+1}\x\right)^2.
\end{aligned}$$
The first equality follows just from definitions -- recall that each $\widehat{x_{ij}}$ stands for  a complex $2\times 2$ matrix. The second equality follows from the rules of computing determinants in $M_{2n} (\mathbb{C})$ using row transformations. For example, the `second row' (which is formed by the third and fourth rows in the respective complex matrix) is obtained by subtracting the `first row' multiplied by 
$\widehat{x_{21}}\cdot\widehat{x_{11}}^{-1}$ from the left. This means that from the third and fourth rows we subtract suitable linear combinations of the first two rows.
Such a transformation preserves the value of determinant. The third equality follows from the rules of computing determinants of block matrices. The fourth one follows from the induction hypothesis and the last one from the defining formula from $(iii)$.

This completes the proof of $(vii)$ for $\widehat{\x}\in H_{n+1}(\Ha_C)$ with $x_{11}\ne0$.

We further deduce that  for $\x\in H_{n+1}(\Ha_C)$ with $x_{11}\ne0$ we have
$$\begin{aligned}
\dt_{n+1}\x&=x_{11}\cdot \dt_n\left(x_{ij}-x_{11}^{-1}x_{i1}x_{1j}\right)_{2\le i,j\le n+1}
\\&=x_{11}^{-n+1}\dt_n\left(x_{11}x_{ij}-x_{i1}x_{1j}\right)_{2\le i,j\le n+1}
\\&= x_{11}^{-n+1}\sum_{\sigma,\pi\in\mathbb{S}_n}
\alpha_{\sigma,\pi} \prod_{k=1}^n (x_{11} x_{\sigma(k)+1,\pi(k)+1}-x_{\sigma(k)+1,1}x_{1,\pi(k)+1}).
\end{aligned}$$
The first equality is just the formula from $(iii)$, next we use the induction hypothesis (conditions $(v)$ and $(vi)$). 

Let us introduce the following notation. If $A\subset \{1,\dots,n\}$, $k\in\{1,\dots,n\}$ and $\sigma,\pi\in\mathbb{S}_n$, we set
$$z^A_{\sigma,\pi,k}=\begin{cases} x_{\sigma(k)+1,\pi(k)+1},& k\in\{1,\dots,n\}\setminus A,\\
x_{\sigma(k)+1,1}\cdot x_{1,\pi(k)+1}, & k\in A.
\end{cases}$$
Then for $\x\in H_{n+1}(\Ha_C)$ with $x_{11}\ne0$ we have
$$\begin{aligned}
\dt_{n+1}\x&= \sum_{m=0}^n (-1)^{m} x_{11}^{-m+1}\sum_{\sigma,\pi\in\mathbb{S}_n}\alpha_{\sigma,\pi}\sum_{A\subset \{1,\dots,n\}, \abs{A}=m } \prod_{k=1}^n z^A_{\sigma,\pi,k}.
\end{aligned}$$
For $m\in\{0,\dots,n\}$ set
$$P_m((x_{ii})_{2\le i\le n+1},(x_{ij})_{1\le i,j\le n+1, i\ne j})=\sum_{\sigma,\pi\in\mathbb{S}_n}\alpha_{\sigma,\pi}\sum_{A\subset \{1,\dots,n\}, \abs{A}=m } \prod_{k=1}^n z^A_{\sigma,\pi,k},$$
so
$$\dt_{n+1}\x= \sum_{m=0}^n (-1)^{m} x_{11}^{-m+1}P_m((x_{ii})_{2\le i\le n+1},(x_{ij})_{1\le i,j\le n+1, i\ne j})\mbox{ if }x_{11}\ne0.$$

Now fix for a while $\x\in H_{n+1}(\Ha_C)$. If we apply the last equality to $\lambda\cdot \1+\x$ with $\lambda + x_{11}\ne0$ in place of $\x$ we obtain
$$\dt_{n+1} (\lambda\cdot \1+\x)=\sum_{m=0}^{n} (-1)^{m} (\lambda+x_{11})^{-m+1}P_m((\lambda+x_{ii})_{2\le i\le n+1},(x_{ij})_{1\le i,j\le n+1, i\ne j}),$$ for all $\lambda\in\ce\setminus\{-x_{11}\}$.
It follows from the definition of $P_m$ that the function 
$$\lambda\mapsto P_m((\lambda+x_{ii})_{2\le i\le n+1},(x_{ij})_{1\le i,j\le n+1, i\ne j})$$
is a polynomial with coefficients in $\Ha_C$.  We observe that its values are in fact complex numbers. Indeed, by the already proved part of $(iv)$ we get that
for each $\x\in H_{n+1}(\Ha_C)$ and each $\lambda\in\ce$ with $\lambda\ne -x_{11}$, we have 
$$\sum_{m=0}^{n} (-1)^{m} (\lambda+x_{11})^{-m+1}P_m((\lambda+x_{ii})_{2\le i\le n+1},(x_{ij})_{1\le i,j\le n+1, i\ne j}) 
\in\ce.$$
We deduce that 
$$\forall\lambda,\mu\in\ce\colon \mu\ne 0
\Rightarrow\sum_{m=0}^{n} (-1)^{m} \mu^{-m+1}P_m((\lambda+x_{ii})_{2\le i\le n+1},(x_{ij})_{1\le i,j\le n+1, i\ne j})\in\ce.$$
It follows that for each $m$ we have
$$\forall\lambda\in\ce\colon P_m((\lambda+x_{ii})_{2\le i\le n+1},(x_{ij})_{1\le i,j\le n+1, i\ne j})\in\ce,$$
hence 
$$\lambda\mapsto P_m((\lambda+x_{ii})_{2\le i\le n+1},(x_{ij})_{1\le i,j\le n+1, i\ne j})$$
is a polynomial with complex coefficients. 
Thus 
$$\lambda\mapsto\sum_{m=0}^{n} (-1)^{m} (\lambda+x_{11})^{-m+1}P_m((\lambda+x_{ii})_{2\le i\le n+1},(x_{ij})_{1\le i,j\le n+1, i\ne j})$$
is a rational function. On the other hand,
by the already proved part of $(vii)$ we know that
$$\det(\lambda\widehat{\1}+\widehat{\x})=\left(\dt_{n+1}(\lambda\1+\x)\right)^2,\quad \lambda\in\ce\setminus\{-x_{11}\}.$$
But it is known that $\lambda\mapsto \det(\lambda\widehat{\1}+\widehat{\x})$ is a polynomial of degree $2n$. It follows that
$\lambda\mapsto \dt_n(\lambda\cdot\1+\x)$ must be a polynomial of degree $n$ (as it is a rational function whose square is a polynomial of degree $2n$). Since polynomials are continuous on $\ce$, this allows us to define $\dt_{n+1}\x$ in case $x_{11}=0$ by 
$$\dt_{n+1}\x=\lim_{\lambda\to0} \dt_{n+1}(\lambda\1+\x).$$
Thus $\dt_{n+1}$ is now defined on all $H_{n+1}(\Ha_C)$ and assertion $(vii)$ holds for $n+1$.

Next we are going to prove that $P_m$ is  constant zero for $m\ge2$. Assume not. Let $m_0$ be the greatest number for which $P_{m_0}$ is not constant zero and assume $m_0\ge2$.

We already know that
$$\lambda\mapsto \dt_{n+1} (\lambda\cdot \1+\x)=\sum_{m=0}^{m_0} (-1)^{m} (\lambda+x_{11})^{-m+1}P_m((\lambda+x_{ii})_{2\le i\le n+1},(x_{ij})_{1\le i,j\le n+1, i\ne j})
$$
is a complex polynomial of degree  $n$. It is equal to
$$(\lambda+x_{11})^{-m_0+1}\sum_{m=0}^{m_0} (-1)^{m} (\lambda+x_{11})^{m_0-m}P_m((\lambda+x_{ii})_{2\le i\le n+1},(x_{ij})_{1\le i,j\le n+1, i\ne j}),$$
so $-x_{11}$ must be a root of the sum. By plugging $-x_{11}$ we see that it must be also a root of the polynomial
$$\lambda\mapsto P_{m_0}((\lambda+x_{ii})_{2\le i\le n+1},(x_{ij})_{1\le i,j\le n+1, i\ne j}).$$ But this should hold for any value of $x_{11}$, so we get
$$P_{m_0}((\lambda+x_{ii})_{2\le i\le n+1},(x_{ij})_{1\le i,j\le n+1, i\ne j})=0$$
for all $\lambda\in\ce$. In particular by plugging $\lambda=0$ we conclude that $P_{m_0}$ is constant zero. This contradiction completes the proof of $(vi)$ for $n+1$.

Assertion $(v)$ follows immediately from $(vi)$. We also have the continuity of $\dt_{n+1}$.
Assertion $(viii)$ follows also from $(vi)$ using standard algebraic expansion.

This completes the induction argument.

Let us continue by proving assertion $(ix)$. Let $\x\in H_{n}(\Ha_C)$. By $(viii)$ we know that 
$$\lambda\mapsto \dt_n(\lambda\cdot\1-\x)$$
is a polynomial of degree $n$ with coefficient $\dt_n\1=1$ at $\lambda^n$ and coefficient $\dt_n(-\x)=(-1)^n\dt_n \x$ at $\lambda^0$ (the last equality follows from $(v)$). Hence, this polynomial has $n$ complex roots (counted with their multiplicities) and their product is $\dt_n\x$.

Moreover, by $(vii)$ we get
$$\det(\lambda\cdot \widehat{\1}-\widehat{\x})=\left(\dt_n(\lambda\cdot\1-\x)\right)^2,$$
so the roots of the polynomial $\lambda\mapsto \dt_n(\lambda\cdot\1-\x)$ are exactly the eigenvalues of $\widehat{\x}$ and their multiplicities as eigenvalues are  exactly twice their multiplicites as roots. This completes the proof. 
\end{proof}

 For $n=1,2,3$ there are easy formulae for $\dt_n$ similar to the classical formulae for determinant of complex matrices. The formula for $\dt_1$ is contained already in the above theorem. Let us explicitly show the formulae for $n=2$ and $n=3$

 Assume that $\x\in H_2(\Ha_C)$ such that $x_{11}\ne0$. Then
$$\dt_2\x=x_{11}\cdot\dt_1(x_{22}-x_{11}^{-1}x_{21}x_{12})=x_{11}\cdot(x_{22}-x_{11}^{-1}x_{21}x_{12})
=x_{11}x_{22}-x_{21}x_{12}.$$
The last expression is a continuous mapping defined on all $H_2(\Ha_C)$ and is completely analogous to the classical one. Note that $\Ha_C$ is not commutative, but the two products in the last expression do commute as $x_{11},x_{22}\in\ce$ and $x_{21}=x_{12}^\inv$.

Next assume that $\x\in H_3(\Ha_C)$ with $x_{11}\ne0$. Then, by the just obtained formula for $n=2$ we get 
$$\begin{aligned}\dt_3\x&=x_{11}\cdot\dt_2\begin{pmatrix}
x_{22}-x_{11}^{-1}x_{21}x_{12}&x_{23}-x_{11}^{-1}x_{21}x_{13}\\ x_{32}- x_{11}^{-1}x_{31}x_{12}&x_{33} - x_{11}^{-1}x_{31}x_{13}
\end{pmatrix}
\\&=x_{11}\cdot((x_{22}-x_{11}^{-1}x_{21}x_{12})(x_{33}-x_{11}^{-1}x_{31}x_{13})\\&\qquad\qquad-(x_{32}-x_{11}^{-1}x_{31}x_{12})(x_{23}-x_{11}^{-1}x_{21}x_{13}))
\\&=x_{11}x_{22}x_{33}- x_{11} x_{22} x_{31} x_{11}^{-1} x_{13}-x_{21}x_{12}x_{33}+ x_{21}x_{12}x_{11}^{-1} x_{31}x_{13}
\\&\qquad -x_{11}x_{32}x_{23}+ x_{11} x_{32} x_{11}^{-1} x_{21}x_{13} +x_{31}x_{12}x_{23}-x_{31}x_{12}x_{11}^{-1} x_{21}x_{13}
\\&=x_{11}x_{22}x_{33} +x_{32}x_{21}x_{13}+x_{31}x_{12}x_{23}\\&\qquad
-x_{11}x_{32}x_{23}-x_{22}x_{31}x_{13}-x_{21}x_{12}x_{33}
\end{aligned}$$
because $x_{11}\in \mathbb{C}$ and by Proposition \ref{P:real octonions}  
$$x_{31}x_{12}x_{21}x_{13}=x_{31}x_{12}x_{12}^\inv x_{13}
=x_{12}x_{12}^\inv x_{31} x_{13}
=x_{12}^\inv x_{12} x_{31} x_{13}
=x_{21} x_{12} x_{31} x_{13}.$$
Hence we have an analogy of the classical Sarrus' rule. The reader should be warned that $\Ha_C\cong M_2$ is associative but not commutative, and so the order of the products in the previous expression is decisive. 

Now fix $n\in\en$ and a unitary element $\e\in H_n(\Ha_C)$. We may consider on $H_n(\Ha_C)$ a structure of JB$^*$-algebra with unit $\e$. Recall that the operations are defined by
$$\x\circ_{\e}\y=\J {\x}{\e}{\y}=\frac12 (\x\e^*\y+\y\e^*\x) \mbox{ and }\x^{*_{\e}}=\J{\e}{\x}{\e}=\e\x^*\e$$
for $\x,\y\in H_n(\Ha_C)$.

Moreover, ${\e}$ is a unitary element also in $M_{n}(\Ha_c)$, hence we may define  such a structure of JB$^*$-algebra on $M_{n}(\Ha_C)$ as well. We may further equip $M_{n}(\Ha_C)$ with a structure of a C$^*$-algebra with unit $\e$ if we consider the same involution and define an associative multiplication by
$$\x\cdot_{\e}\y=\x \e^* \y,\quad \x,\y\in M_{n}(\Ha_C).$$

\begin{lemma}\label{L:isomorphism e-1}
Let $\e$ be a unitary element of $H_n(\Ha_C)$.
Then there is a mapping $T:M_{n}(\Ha_C)\to M_{n}(\Ha_C)$ with the properties
\begin{enumerate}[$(a)$]
    \item $T$ is a $*$-isomorphism of $M_{n}(\Ha_C)$ equipped with the product $\cdot_{\e}$ and the involution $^{*_{\e}}$ onto $M_{n}(\Ha_C)$ equipped with the standard structure of  C$^*$-algebra.
    \item $T(\x^\inv)=T(\x)^\inv$ for $\x\in\Ha_C$. In particular, $T$ maps $H_n(\Ha_C)$ onto $H_n(\Ha_C)$.
\end{enumerate}
\end{lemma}
 
 \begin{proof}
 Let $\vv\in H_n(\Ha_C)$ be a unitary element such that $\vv^2=\e$. Such $\vv$ exists by Lemma~\ref{L:sqrt}$(b)$. It is clear that $\vv$ commutes with $\e$ in $M_n(\Ha_C)$.
It is enough to set
 $$T(\x)=\vv^*\x\vv^*,\quad \x\in M_n(\Ha_C).$$
 It is clearly a linear bijection, Moreover,
 $$T(\x\cdot_{\e}\y)=T(\x\e^*\y)=\vv^*\x\e^*\y\vv^*=T(\x) T(\y)$$
 as $\e^*=(\vv^*)^2$. Further,
 $$T(\x^{*_{\e}})=T(\e\x^*\e) = \vv^* \e\x^*\e \vv^* =\vv^*\vv \vv \x^* \vv \vv \vv^*=\vv\x^*\vv=(T\x)^*.$$
 
Finally,
$$T(\x^\inv)=\vv^*\x^\inv\vv^*=(T\x)^\inv.$$
This completes the proof.
 \end{proof}

\begin{prop}\label{P:product-dtn}
Let $\e$ be a unitary element in $H_n(\Ha_C)$. Let $T$ be a mapping with the properties from Lemma~\ref{L:isomorphism e-1}. Let us define
$$\dt_{n,\e}\x=\dt_n(T\x),\quad\x\in H_n(\Ha_C).$$
Then the following holds:
\begin{enumerate}[$(i)$]
    \item The function $\dt_{n,\e}$ does not depend on the concrete choice of $T$.
    \item $\dt_n\x=\dt_{n,\e}\x\cdot \dt_n \e$ for $\x\in H_n(\Ha_C)$.
\end{enumerate}
\end{prop} 

\begin{proof}
Let us fix any $T$ satisfying the properties from Lemma~\ref{L:isomorphism e-1}. Then $\dt_nT\x$ is the product of all the roots of the polynomial $\lambda\mapsto\dt_n(\lambda\1-T\x)$, each one counted with its multiplicity (this follows from Theorem~\ref{T:determinant}$(viii)$). Moreover, by Theorem~\ref{T:determinant}$(vii)$ we get
$$(\dt_n(\lambda \cdot\1-T\x))^2=\det(\lambda\widehat{\1}-\widehat{Tx})=\det(T(\lambda\widehat{\e}-\widehat{x})).$$
The right-hand side does not depend on the choice of $T$. Indeed, assume that $T_1$ and $T_2$ are two possible choices of $T$. Then $T_2T_1^{-1}$ is a $*$-automorphism of $M_n(\Ha_C)=M_{2n}$. It follows that there is a unitary matrix $\uu\in M_{2n}$ such that $T_2T_1^{-1}\a=\uu^*\a\uu$ for $\a\in M_{2n}$ (see, for example, \cite[Corollary 5.42]{DouglasBook1998} or the celebrated Schur theorem \cite[Theorem 10.2.2]{FlemingJamisonBookVol2}). In particular, $T_2T_1^{-1}$ preserves the determinant, hence
$$\det(T_1\a)=\det(T_2T_1^{-1}T_1\a)=\det(T_2\a)$$
for any $\a\in M_{2n}$. 

Thus, we may assume that $T$ is of the form from the proof of Lemma~\ref{L:isomorphism e-1}. Then
$$\begin{aligned}
\det(T(\lambda\widehat{\e}-\widehat{x}))&=\det(\vv^*(\lambda\widehat{\e}-\widehat{x})\vv^*)=
(\det(\vv^*))^2\det(\lambda\widehat{\e}-\widehat{x})\\&=(\det(\vv^*))^2(\dt_n(\lambda\e-\x))^2.\end{aligned}$$

We deduce
$$(\dt_n(\lambda \cdot\1-T\x))^2=(\det(\vv^*))^2(\dt_n(\lambda\e-\x))^2,$$
so $\dt_n T\x$ is the product of all roots of the polynomial $\lambda\mapsto\dt_n(\lambda\e-\x)$, each one counted with its multiplicity.

By Theorem~\ref{T:determinant}$(viii)$ we know that
$\lambda\mapsto\dt_n(\lambda\e-\x)$ is a polynomial of degree at most $n$ with the coefficient $\dt_n\e$ at $\lambda^n$ and coefficient $\dt_n(-\x)=(-1)^n\dt_n\x$ at $\lambda^0$. Note that $\dt_n\e\ne0$ (it is a complex unit as $(\dt_n\e)^2=\det\widehat{\e}$ is a complex unit being the determinant of a unitary matrix). It follows that the product of the roots (which equals to $\dt_n T\x=\dt_{n,\e}\x$) is $(\dt_n\e)^{-1}\dt_n\x$. This completes the proof.
\end{proof}

\section{Automorphisms of $\O$ and $\O_R$}\label{sec:auto-oct}

By an \emph{automorphism} of $\O$ we mean a linear bijection $T:\O\to\O$ which preserves the structure of $\O$, i.e., which satisfies the following properties:
\begin{enumerate}[$(i)$]
    \item $T(1)=1$,
    \item $T(x^*)=T(x)^*$ for $x\in\O$,
    \item $T(x^\inv)=T(x)^\inv$ for $x\in\O$,
    \item $T(x\boxdot y)=T(x)\boxdot T(y)$ for $x,y\in\O$.
\end{enumerate}
If $T$ is an automorphism of $\O$, then we get
$$T(\overline{x})=T((x^\inv)^*)=(T(x)^\inv)^*=\overline{T(x)}$$  for each $x\in\O$,
and thus $T$ preserves $\O_R$. Then $T|_{\O_R}$ is an \emph{automorphism} of $\O_R$, i.e., a (real-)linear bijection of $\O_R$ onto $\O_R$ satisfying conditions $(i)-(iv)$ for $x,y\in\O_R$ (note that in this case conditions $(ii)$ and $(iii)$ coincide).

Conversely, if $T$ is an automorphism of $\O_R$, the mapping
$$\widetilde{T}(x+iy)=T(x)+iT(y), \quad x,y\in\O_R$$
is clearly an automorphism of $\O$.

We will need also a weaker version of automorphism. A linear bijection $T:\O\to \O$ will be called an \emph{asymmetric triple isomorphism} if
$$T((x\boxdot y^*)\boxdot z)=(T(x)\boxdot T(y)^*)\boxdot T(z)\mbox{ for }x,y,z\in\O.$$
If it satisfies moreover $T(\overline{x})=\overline{T(x)}$ for $x\in\O$, it will be called a \emph{hermitian asymmetric triple isomorphism}.

If $T$ is a hermitian asymmetric triple isomorphism of $\O$, then $T$ preserves $\O_R$ and hence $T|_{\O_R}$ is an asymmetric triple isomorphism of $\O_R$, i.e., a (real-)linear bijection of $\O_R$ satisfying
$$T((x\boxdot y^\inv)\boxdot z)=(T(x)\boxdot T(y)^\inv)\boxdot T(z)\mbox{ for }x,y,z\in\O_R.$$

Conversely, if $T$ is an asymmetric triple isomorphism of $\O_R$, the mapping
$$\widetilde{T}(x+iy)=T(x)+iT(y), \quad x,y\in\O_R$$
is clearly a hermitian asymmetric triple isomorphism of $\O$.

We will later need the results on hermitian asymmetric triple isomorphism on $\O$, but due to the easy correspondence explained above we restrict ourselves to $\O_R$.

The following observation is easy:

\begin{obs}\label{obs:auto}
Let $T:\O_R\to\O_R$ be a linear bijection. Then $T$ is an automorphism if and only if it is an asymmetric triple automorphism and $T(1)=1$.
\end{obs}

We denote by $e_0,e_1,\dots,e_7$ the canonical basis of $\O_R$ (note that $e_0=1$). We have the following result.

\begin{prop}\label{P:OR-transform}
Let $u\in\O_R$ be any nonzero element. 
\begin{enumerate}[$(a)$]
    \item There is an asymmetric triple isomorphism $T:\O_R\to\O_R$ such that $T(u)\in\er=\span\{e_0\}$;
    \item There is an automorphism $T:\O_R\to\O_R$ such that $T(u)\in\span\{e_0,e_1\}$;
    \item There is an automorphism $T:\O_R\to\O_R$ such that
    $T(e_1)=e_1$ and
    $T(u)\in\span\{e_0,e_1,e_2\}$.
\end{enumerate}
\end{prop}

To prove it we will use two lemmata:

\begin{lemma}\label{L:OR-mulipl}
Let $h_1,h_2\in\mathbb{H}$ be two quaternions of norm $1$. Then the mapping
$T:\O_R\to\O_R$ defined by
$$T(x_1,x_2)=(x_1h_1,x_2h_2)$$
is an asymmetric triple isomorphism. If $h_1=1$, it is even an automorphism.
\end{lemma}

\begin{proof}
We have
$$x\boxdot y^\inv=(x_1,x_2)\boxdot(y_1^\inv,-y_2)=(x_1 y_1^\inv+y_2 x_2^\inv,-x_1^\inv y_2+y_1^\inv x_2)$$
and so
$$(x\boxdot y^\inv)\boxdot z=(x_1 y_1^\inv z_1+y_2 x_2^\inv z_1+z_2y_2^\inv x_1-z_2x_2^\inv y_1,y_1x_1^\inv z_2+x_2y_2^\inv z_2-z_1x_1^\inv y_2+z_1y_1^\inv x_2).$$
In view of the associativity of the multiplication of quaternions and using the fact that $h_1h_1^\inv=h_2h_2^\inv=1$ it easily follows that $T$ is an asymmetric triple isomorphism.

If $h_1=1$, then clearly $T(1)=1$, so $T$ is an automorphism by Observation~\ref{obs:auto}.
\end{proof}

\begin{lemma}\label{L:OR-permutation}
The  linear bijections of $\O_R$ defined by
$$\begin{aligned}
P_1&:e_0\mapsto e_0,e_1\mapsto e_1,e_2\mapsto e_7,e_3\mapsto e_6, e_4\mapsto e_2, e_5\mapsto e_3, e_6\mapsto e_5,e_7\mapsto e_4.
\\
P_2&:e_0\mapsto e_0,e_1\mapsto e_4, e_2\mapsto e_7,e_3\mapsto e_3,e_4\mapsto e_6,e_5\mapsto e_2,e_6\mapsto e_1,e_7\mapsto e_5\end{aligned}$$
are automorphisms of $\O_R$.
\end{lemma}

\begin{proof}
Properties $(i)$ and $(iii)$ are easy.  It is enough to prove property $(iv)$ for the basic vectors. This  can be proved using the following description of the multiplication (see \cite[p. 235]{LoSo}):

We have $e_0\boxdot e_j=e_j\boxdot e_0=e_j$ for each $j\in\{0,\dots,7\}$ and $e_j^2=-e_0$ for $j\in\{1,\dots,7\}$. The multiplication of the remaining pairs of  basic vectors is done using the list
$$132,154,167,264,275,356,374.$$
For example, the first scheme, $132$, means that
$$e_1\boxdot e_3=e_2, e_3\boxdot e_2=e_1, e_2\boxdot e_1=e_3$$
and the products in the opposite order have opposite sign.

It is now easy to check that the two permutations preserve this schemes.
\end{proof}

\begin{proof}[Proof of Proposition~\ref{P:OR-transform}]
$(a)$ Let $u=(u_1,u_2)$. Find two quaternions $h_1,h_2$ of norm one such that $u_1h_1,u_2h_2\in\er$. Let $T_1(x)=(x_1h_1,x_2h_2)$. Then $T_1(u)\in\span\{e_0,e_4\}$. Let $P_1$ be the automorphism from Lemma~\ref{L:OR-permutation} Then $P_1T_1(u)\in\span\{e_0,e_2\}$. In particular, it is a quaternion. Find a quaternion $h_3$ of norm one such that $h_3P_1T_1(u)\in\er$. Take $T_2(x)=(x_1h_3,x_2)$. Then we may take $T=T_2P_1T_1$.

$(b)$ Let $u=(u_1,u_2)$. Find a quaternion $h_1$ of norm one such that $u_2h_1\in\er$. Set $T_1(x)=(x_1,x_2h_1)$. Let $P_2$ be from Lemma~\ref{L:OR-permutation}. Then $P_2T_1(u)=(v_1,v_2)$, where $v_1\in\span\{e_0,e_3\}$. Find a quaternion $h_2$ of norm one such that $v_2h_2$ is a multiple of $e_3$. Set $T_2(x)=(x_1,x_2h_2)$. Then $T_2P_2T_1(u)\in\span\{e_0,e_3,e_7\}$, hence $P_1T_2P_2T_1(u)\in\span\{e_0,e_4,e_6\}$, i.e. $P_1T_2P_2T_1(u)=(w_1,w_2)$, where $w_1\in\er$ and $w_2\in\span\{1,e_2\}$. Find a quaternion $h_3$ of norm one such that $w_2h_3\in\er$. Set $T_3(x)=(x_1,x_2h_3)$. It is now enough to set
$$T=P_1T_3P_1T_2P_2T_1.$$

$(c)$ Let $u=(u_1,u_2)$. Find a quaternion $h_1$ of norm one such that $u_2h_1$ is a multiple of $e_3$. Set $T_1(x)=(x_1,x_2h_1)$. Let $P_1$ be from Lemma~\ref{L:OR-permutation}.Then $P_1T_1(u)=(v_1,v_2)$, where $v_1\in\span\{e_0,e_1\}$.  Find a quaternion $h_2$ of norm one such that $v_2h_2\in\er$. Set $T_2(x)=(x_1,xh_2)$. It is enough to take
$$T=P_1T_2P_1T_1.$$
\end{proof}

\begin{remark}
It was pointed out by the referee that assertions $(b)$ and $(c)$ of Proposition~\ref{P:OR-transform} follow also from \cite[Corollary 1.7.5]{springer-veldcamp}. The quoted abstract result deals with composition algebras over a general field, so its application requires some explanation. Firstly, $\O_R$ is a composition algebra over $\er$ in the sense of \cite[Definition 1.2.1]{springer-veldcamp}. The quadratic form $N$ from the definition is just the norm on $\O_R$ (recall that it coincides with the euclidean norm $\norm{\cdot}_2$) and the respective bilinear form (denoted by $\ip{a}{b} := N(a+b) -N(a) -N(b)$ in \cite{springer-veldcamp}) equals the double of the inner product. This follows for example from \cite[Proposition 1.5.3]{springer-veldcamp} applied twice -- at first to $\ce$ considered as a commutative and associative composition algebra over $\er$ and then to the resulting algebra of quaternions (in both cases we take $\lambda=-1$).

Now we may see that assertion $(b)$ of Proposition~\ref{P:OR-transform} follows from \cite[Corollary 1.7.5]{springer-veldcamp} applied to $a=u-\ip u{e_0} e_0$ (where $\ip{\cdot}{\cdot}$ denotes the inner product as in Section~\ref{sec:CD-C6}) and $a'=\norm{a}\cdot e_1$ (and $b=b'=0$). To get assertion $(c)$ we set
$a=a'=e_1$, $b=u-\ip u{e_0}e_0-\ip u{e_1}e_1$ and $b'=\norm{b}\cdot e_2$.
\end{remark} 

\section{Automorphism of $C_6$}\label{sec:auto-c6}

We start by observing that the mapping exchanging two rows and subsequently the two corresponding columns is a Jordan $*$-isomorphism.

\begin{lemma}\label{L:exchanging}
Let $k,l\in \{1,2,3\}$ be two distinct numbers. For $\x\in C_6$ let $U_{k,l}(\x)$ be the matrix made from $\x$ by exchanging the $k$-th and $l$-th rows and subsequently exchanging the $k$-th and $l$-th columns. Then $U_{k,l}$ is a Jordan $*$-automorphism of $C_6$. 
\end{lemma}

\begin{proof}
Let us prove it for $k=1$ and $l=2$. The remaining cases are analogous. Let
$$\uu=\begin{pmatrix}
0&1&0\\1&0&0\\0&0&1
\end{pmatrix}.$$
Then $\uu$ is a symmetry (i.e., a self-adjoint unitary element) in $C_6$, thus $\x\mapsto \J {\uu}{\x^*}{\uu}$ is a triple automorphism of $C_6$ by Lemma~\ref{L:shift}. Moreover, it maps $\1$ to $\1$, so it is a Jordan $*$-automorphism.

It remains to compute that $U_{1,2}(\x)=\J {\uu}{\x^*}{\uu}$. So, let us compute.
$$\J {\uu}{\x^*}{\uu}=2 \uu\circ(\x\circ \uu)-(\uu\circ \uu)\circ \x=2 \uu\circ(\x\circ \uu)-\x$$
as $\uu\circ \uu=\1$. Assume
$$\x=\begin{pmatrix}
\alpha & a & b \\ a^\inv & \beta & c \\ b^\inv & c^\inv & \gamma
\end{pmatrix},$$
where $\alpha,\beta,\gamma\in\ce$ and $a,b,c\in\O$. Then
$$\begin{aligned}
\uu\circ \x&=\frac12(\uu\boxdot \x+\x\boxdot \uu) =\frac12\left(
\begin{pmatrix} a^\inv & \beta & c \\\alpha & a & b \\ b^\inv & c^\inv & \gamma
\end{pmatrix} +
\begin{pmatrix}
a &\alpha  & b \\ \beta& a^\inv & c \\ c^\inv & b^\inv & \gamma
\end{pmatrix}
\right)
\\&=\begin{pmatrix}
\frac12(a^\inv+a) &\frac12(\alpha+\beta) &\frac12(b+c) \\ \frac12(\alpha+\beta) &\frac12(a+a^\inv) &\frac12(b+c) \\
\frac12(b^\inv+c^\inv) & \frac12(b^\inv+c^\inv) & \gamma
\end{pmatrix},
\end{aligned}$$
hence
$$\begin{aligned}
2\uu\circ(\uu\circ \x)&=\uu\boxdot(\uu\circ \x)+(\uu\circ \x)\boxdot \uu
\\&= \begin{pmatrix}
 \frac12(\alpha+\beta) &\frac12(a+a^\inv) &\frac12(b+c) \\\frac12(a^\inv+a) &\frac12(\alpha+\beta) &\frac12(b+c) \\
\frac12(b^\inv+c^\inv) & \frac12(b^\inv+c^\inv) & \gamma
\end{pmatrix}\\&\qquad\qquad +
\begin{pmatrix}
\frac12(\alpha+\beta) &\frac12(a^\inv+a) &\frac12(b+c) \\ \frac12(a+a^\inv) &\frac12(\alpha+\beta) &\frac12(b+c) \\
\frac12(b^\inv+c^\inv) & \frac12(b^\inv+c^\inv) & \gamma
\end{pmatrix} \\&=\begin{pmatrix}
\alpha+\beta & a^\inv+a & b+c \\ a+a^\inv & \alpha+\beta & b+c \\ b^\inv+c^\inv & b^\inv+c^\inv & 2\gamma
\end{pmatrix}.
\end{aligned}$$
Thus
$$\J {\uu}{\x^*}{\uu}=2 \uu\circ(\x\circ \uu)-\x
=\begin{pmatrix}
\beta & a^\inv & c \\
a &\alpha & b \\ c^\inv& b^\inv &\gamma
\end{pmatrix}=U_{1,2}(\x).$$
\end{proof}

\begin{prop}\label{P:automorphisms-C6}
Let $T$ be a hermitian asymmetric triple isomorphism of $\O$. 
Define $\widetilde{T}:C_6\to C_6$ by
$$\widetilde{T}\begin{pmatrix}
\alpha & a & b \\ a^\inv & \beta & c \\ b^\inv & c^\inv & \gamma
\end{pmatrix}=\begin{pmatrix}
\alpha & T(a) & T(b)\boxdot T(1)^\inv \\ T(a)^\inv & \beta & T(c^\inv)^\inv \\ T(1)\boxdot T(b)^\inv & T(c^\inv) & \gamma
\end{pmatrix}.$$
Then $\widetilde{T}$ is a Jordan $*$-automorphism of $C_6$.
\end{prop}

\begin{proof}
It is clear that $\widetilde{T}$ is a linear mapping. Further, $T(1)$ is a non-zero element of $\O_R$, so we may use Proposition~\ref{P:real octonions}$(iii)$ to deduce that $\widetilde{T}$ is one-to-one. Therefore, $\widetilde{T}$ is a linear bijection of $C_6$.

Since $T(1)\in\O_R$ and $T$ commutes with the conjugation, we get that $\widetilde{T}(\x^*)=\widetilde{T}(\x)^*$ for each $\x\in C_6$. Hence, it remains to prove that $\widetilde{T}(\x^2)=\widetilde{T}(\x)^2$ for $\x\in C_6$.

First observe that $T$ is a unitary operator on $\O$ equipped with the underlying Hilbert space structure (as it is a triple automorphism commuting with the conjugation cf. Proposition \ref{P:octonions}$(iii)$).

Let us fix 
$$\x=\begin{pmatrix}
\alpha & a & b \\ a^\inv & \beta & c \\ b^\inv & c^\inv & \gamma
\end{pmatrix}\in C_6.$$
Then
$$\x^2=\begin{pmatrix}
\alpha^2+a\boxdot a^\inv+b\boxdot b^\inv & \alpha a+\beta a+b\boxdot c^\inv & \alpha b+a\boxdot c+\gamma b \\ \alpha a^\inv+\beta a^\inv+c\boxdot b^\inv & a^\inv\boxdot a+\beta^2+c\boxdot c^\inv & a^\inv\boxdot b+\beta c+\gamma c \\
\alpha b^\inv+ c^\inv\boxdot a^\inv+\gamma b^\inv & b^\inv\boxdot a+\beta c^\inv+\gamma c^\inv & b^\inv\boxdot b+c^\inv \boxdot c+\gamma^2
\end{pmatrix}$$
We are going to prove that $\widetilde{T}(\x^2)=\widetilde{T}(\x)^2$ entrywise. Let us first look at the diagonal. By the definition of $\widetilde{T}$ the diagonal of $\widetilde{T}(\x^2)$ coincides with the diagonal of $\x^2$. Hence, we have to prove that
\begin{equation*}
\begin{aligned}
\alpha^2+a\boxdot a^\inv+b\boxdot b^\inv&=\alpha^2+T(a)\boxdot T(a)^\inv + (T(b)\boxdot T(1)^\inv)\boxdot (T(1)\boxdot T(b)^\inv),
\\ a^\inv\boxdot a+\beta^2+c\boxdot c^\inv&= T(a)^\inv \boxdot T(a)+\beta^2+T(c^\inv)^\inv\boxdot T(c^\inv),\\
b^\inv\boxdot b+c^\inv \boxdot c+\gamma^2&=(T(1)\boxdot T(b)^\inv)\boxdot (T(b)\boxdot T(1)^\inv)+T(c^\inv)\boxdot T(c^\inv)^\inv+\gamma^2.
\end{aligned}\end{equation*}
To prove these equalities we will use the following computation:
$$\begin{aligned}
T(a)\boxdot T(a)^\inv&=\ip{T(a)}{\overline{T(a)}}=\ip{T(a)}{T(\overline{a})}=\ip{a}{\overline{a}}=a\boxdot a^\inv,
\end{aligned}$$
The first equality follows from Lemma~\ref{L:CD}$(a)$, the second one follows from the fact that $T$ commutes with the conjugation. In the third one we use that $T$ is a unitary operator. The last equality follows again from Lemma~\ref{L:CD}$(a)$.

If we apply the just proved equality to $c^\inv$ in place of $a$ we get
$$T(c^\inv)\boxdot T(c^\inv)^\inv=c^\inv\boxdot c.$$
Further,
$$\begin{aligned}
(T(1)\boxdot T(b)^\inv)\boxdot (T(b)\boxdot T(1)^\inv)&=
(T(b)\boxdot T(1)^\inv)^\inv\boxdot (T(b)\boxdot T(1)^\inv)
\\&=\ip{T(b)\boxdot T(1)^\inv}{\overline{T(b)\boxdot T(1)^\inv}}\\&=\ip{T(b)\boxdot T(1)^\inv}{T(\overline{b})\boxdot T(1)^\inv}
\\&=\ip{T(b)}{(T(\overline{b})\boxdot T(1)^\inv)\boxdot T(1)}
=\ip{T(b)}{T(\overline{b})}
\\&=\ip{b}{\overline{b}}=b\boxdot b^\inv
\end{aligned}
$$
The first equality follows from the basic properties of $\O$, the second one from Lemma~\ref{L:CD}$(a)$. In the third one we use that $T$ commutes with the conjugation and that $T(1)\in\O_R$. The fourth one follows from Proposition~\ref{P:octonions}$(ii)$. The fifth one follows from Proposition~\ref{P:real octonions}$(ii)$ (applied to the real and imaginary parts of $T(\overline{b})$, note that $T(1)\in \O_R$ and $\norm{T(1)}=1$).
The sixth equality holds because $T$ is a unitary operator. The last equality follows again from Lemma~\ref{L:CD}$(a)$.

By combining the above computations we deduce that the diagonals of $\widetilde{T}(\x^2)$ and $\widetilde{T}(\x)^2$ coincide. 
It remains to prove the coincidence of the components of these two elements outside the diagonal. Since they are $^\inv$-hermitian, it is enough to prove they coincide above the diagonal. This reduces to proving three equalities:
$$\begin{aligned}
 T(\alpha a+\beta a+b\boxdot c^\inv) &= \alpha T(a)+\beta T(a)+(T(b)\boxdot T(1)^\inv)\boxdot T(c^\inv), 
\\   T(\alpha b+a\boxdot c+\gamma b)\boxdot T(1)^\inv &= \alpha T(b)\boxdot T(1)^\inv+T(a)\boxdot T(c^\inv)^\inv+\gamma T(b)\boxdot T(1)^\inv,
\\  T((a^\inv\boxdot b+\beta c+\gamma c)^\inv)^\inv &=T(a)^\inv\boxdot (T(b)\boxdot T(1)^\inv)+\beta T(c^\inv)^\inv+\gamma T(c^\inv)^\inv.
\end{aligned}$$
To prove them we use, in addition to linearity of $T$, the assumption that it is a hermitian asymmetric triple isomorphism and so
$$\begin{aligned}
T(b\boxdot c^\inv)&=T((b\boxdot 1^\inv)\boxdot c^\inv)=(T(b)\boxdot T(1)^\inv)\boxdot T(c^\inv),\\
T(a\boxdot c)\boxdot T(1)^\inv&=T((a\boxdot (c^\inv)^\inv)\boxdot 1)\boxdot T(1)^\inv
=((T(a)\boxdot T(c^\inv)^\inv)\boxdot T(1))\boxdot T(1)^\inv
\\&=T(a)\boxdot T(c^\inv)^\inv,
\\
T((a^\inv\boxdot b)^\inv)^\inv&=T(b^\inv\boxdot a)^\inv = T((1\boxdot b^\inv) \boxdot a)^\inv
= ((T(1)\boxdot T(b)^\inv)\boxdot T(a))^\inv\\&=T(a)^\inv\boxdot (T(b)\boxdot T(1)^\inv).
\end{aligned}$$
Note that in the last equality of the second computation we used, similarly as above, Proposition~\ref{P:real octonions}$(ii)$.

This completes the proof.
\end{proof}

\begin{cor}\label{cor:C6-automorphisms}
Let $T$ be a hermitian asymmetric triple isomorphism of $\O$. 
Define $\widetilde{T}_j:C_6\to C_6$ for  $j=1,2,3$ by
$$\begin{aligned}
\widetilde{T}_1\begin{pmatrix}
\alpha & a & b \\ a^\inv & \beta & c \\ b^\inv & c^\inv & \gamma
\end{pmatrix}&=\begin{pmatrix}
\alpha & T(a) & T(1)\boxdot T(b^\inv) \\ T(a)^\inv & \beta & T(c^\inv)^\inv \\ T(b^\inv)^\inv\boxdot T(1)^\inv & T(c^\inv) & \gamma
\end{pmatrix},\\
\widetilde{T}_2\begin{pmatrix}
\alpha & a & b \\ a^\inv & \beta & c \\ b^\inv & c^\inv & \gamma
\end{pmatrix}&=\begin{pmatrix}
\alpha & T(a^\inv)^\inv & T(b^\inv)^\inv \\ T(a^\inv) & \beta & T(c)\boxdot T(1)^\inv \\  T(b^\inv) & T(1)\boxdot T(c)^\inv & \gamma
\end{pmatrix},\\
\widetilde{T}_3\begin{pmatrix}
\alpha & a & b \\ a^\inv & \beta & c \\ b^\inv & c^\inv & \gamma
\end{pmatrix}&=\begin{pmatrix}
\alpha & T(a)\boxdot T(1)^\inv & T(b) \\ T(1)\boxdot T(a)^\inv & \beta & T(c) \\  T(b)^\inv & T(c)^\inv & \gamma
\end{pmatrix}.
\end{aligned}$$
Then $\widetilde{T}_1,\widetilde{T}_2,\widetilde{T}_3$  are Jordan $*$-automorphisms of $C_6$.
\end{cor}

\begin{proof}
Let $\widetilde{T}$ be the automorphism from Propoisition~\ref{P:automorphisms-C6}. It is enough to observe that
$$\widetilde{T}_1=U_{1,3}\widetilde{T}U_{1,3}, 
\widetilde{T}_2=U_{1,2}\widetilde{T}U_{1,2}
\widetilde{T}_3=U_{2,3}\widetilde{T}U_{2,3},$$
where we use the automorphisms $U_{k,l}$ from Lemma~\ref{L:exchanging}.
\end{proof}

\section{Minimal projection in $C_6$}\label{sec:projections}

The aim of this section is to prove the following theorem which provides an explicit description of minimal projections in $C_6$. The result, which is interesting by itself, provides a newfangled detailed description of these elements with potential applications to improve our understanding of this exceptional Jordan algebra.

\begin{thm}\label{T:minimal projections}
 Minimal projections in $C_6$ are exactly elements of the form
$$\begin{gathered}
\begin{pmatrix}
0 & 0 & 0\\0&0&0\\0&0&1
\end{pmatrix},\begin{pmatrix}
0 & 0 & 0\\0&\alpha&a\\0&a^\inv&\frac1\alpha\norm{a}^2
\end{pmatrix},\mbox{ where }\alpha\in\er\setminus\{0\}, a\in\O_R, \alpha=\alpha^2+\norm{a}^2,\\
  \begin{pmatrix}
    \alpha & a & b \\
    a^\inv &\frac1\alpha \norm{a}^2 & \frac1\alpha a^\inv\boxdot b \\ 
    b^\inv & \frac1\alpha b^\inv \boxdot a & \frac1\alpha \norm{b}^2
    \end{pmatrix}, \mbox{ where }\alpha\in\er\setminus\{0\}, a,b\in\O_R, \alpha=\alpha^2+\norm{a}^2+\norm{b}^2.
    \end{gathered}$$
\end{thm}

Let us comment a bit the above formulae for minimal projections. In the second case the third row may be obtained from the second one by multiplying with $\frac1\alpha a^\inv$ from the left. Similarly, in the third case the second row may be obtained from the first one by multiplying with $\frac1\alpha a^\inv$ from the left and the third row may be obtained from the first one by multiplying with $\frac1\alpha b^\inv$ from the left.

The rest of this section is devoted to proving this theorem. The proof will have two parts -- on one hand we need to show that any minimal projection is of the required form and then we need to show the converse, that any matrix of the prescribed form is a minimal projection.

\begin{proof}[Proof of the necessity.]
Assume that $\q$ is a minimal projection in $C_6$. Firstly, it is a projection, i.e., $\q^*=\q$ and $\q=\q^2=\q\boxdot\q$. The condition $\q^*=\q$ implies that
\begin{equation}\label{eq:matrix q}
    \q=\begin{pmatrix}
\alpha & a & b \\ a^\inv & \beta & c \\ b^\inv & c^\inv & \gamma
\end{pmatrix},\mbox{ where } \alpha,\beta,\gamma\in\er\mbox{ and }a,b,c\in\O_R.\end{equation} 
Further, the condition $\q\boxdot \q=\q$ means that
\begin{equation}\label{eq:q2=q}
    \begin{aligned}
\alpha&=\alpha^2+a\boxdot a^\inv +b\boxdot b^\inv,\\
\beta&=a^\inv\boxdot a +\beta^2+c\boxdot c^\inv,\\ \gamma&=b^\inv\boxdot b+c^\inv\boxdot c+\gamma^2,\\
a&=\alpha a+\beta a+b\boxdot c^\inv,\\
b&=\alpha b+a\boxdot c+\gamma b,\\
c&=a^\inv\boxdot b+\beta c+\gamma c.
\end{aligned}\end{equation}

Recall that for any $y\in \O_R$ we have $y\boxdot y^\inv=\norm{y}^2$ (by Proposition~\ref{P:real octonions}$(i)$). Hence, if $\alpha=0$, the first equality says that $a=b=0$ and thus the first row and the first column of $\q$ are zero. Similarly, if $\beta=0$, then the second row and second column of $\q$ are zero and the analogous statement holds for $\gamma$ as well. Since $\q\ne0$, necessarily one of the numbers $\alpha,\beta,\gamma$ is nonzero. 

To simplify the situation we will first assume that $\alpha\ne0$. This assumption may be done without loss of generality, up to applying one of the automorphisms $U_{1,2}$ or $U_{1,3}$ from Lemma~\ref{L:exchanging}.

Next, $\q$ is minimal. This means that $\J {\q}{\x}{\q}$ is a scalar multiple of $\q$ for each $\x\in C_6$. We will apply it for the  canonical projection
$$\p_1=\begin{pmatrix}
1&0&0\\0&0&0\\0&0&0
\end{pmatrix}$$
in place of $\x$.
 
Let us compute 
$$\J {\q}{\p_1}{\q}=2 \q\circ(\q\circ \p_1)-\q\circ \p_1.$$
We have
$$\p_1\circ \q=\begin{pmatrix}
\alpha & \frac12a & \frac12b \\ \frac12 a^\inv & 0 & 0 \\ \frac12 b^\inv & 0 & 0
\end{pmatrix}$$
and
$$\begin{aligned}
(\p_1\circ& \q)\boxdot \q  = \\
&\begin{pmatrix}
\alpha^2+\frac12 a\boxdot a^\inv+\frac12 b\boxdot b^\inv & \alpha a+\frac12\beta a+\frac12 b\boxdot c^\inv & \alpha b+\frac12 a\boxdot c +\frac12 \gamma b
\\ \frac12\alpha a^\inv &\frac12 a^\inv \boxdot a & \frac12 a^\inv \boxdot b
\\ \frac12 \alpha b^\inv & \frac12 b^\inv \boxdot a & \frac 12 b^\inv\boxdot b 
\end{pmatrix},
\end{aligned}$$
so
\begin{multline*}
2(\p_1\circ \q)\circ \q\\=\begin{pmatrix}
2\alpha^2+ a\boxdot a^\inv+ b\boxdot b^\inv & \frac32\alpha a+\frac12\beta a+\frac12 b\boxdot c^\inv & \frac32\alpha b+\frac12 a\boxdot c +\frac12 \gamma b
\\ \frac32\alpha a^\inv+\frac12\beta a^\inv+\frac12 c\boxdot b^\inv & a^\inv \boxdot a &  a^\inv \boxdot b
\\ \frac32\alpha b^\inv+\frac12 c^\inv\boxdot a^\inv +\frac12 \gamma b^\inv & b^\inv \boxdot a &  b^\inv\boxdot b 
\end{pmatrix}
\\=\begin{pmatrix}
\alpha+\alpha^2 & \alpha a+\frac12a & \alpha b+\frac12b \\ \alpha a^\inv +\frac12 a^\inv & \norm{a}^2 &  a^\inv\boxdot b \\ \alpha b^\inv +\frac12 b^\inv & b^\inv\boxdot a & \norm{b}^2
\end{pmatrix}\end{multline*}
where in the last equality we used \eqref{eq:q2=q} 
and Proposition~\ref{P:real octonions}$(i)$.
Hence,
$$\begin{aligned}
\J {\q}{\p_1}{\q}&=2 \q\circ(\q\circ \p_1)-\q\circ \p_1
= \begin{pmatrix}
\alpha^2 & \alpha a & \alpha b\\
\alpha a^\inv & \norm{a}^2 &  a^\inv \boxdot b
\\ \alpha b^\inv & b^\inv \boxdot a &  \norm{b}^2 
\end{pmatrix}
\end{aligned}$$
Recall that this must be a multiple of $\q$. Since $\alpha\ne0$, it is the $\alpha$-multiple. Hence, \begin{equation}\label{eq:multiples}
    \begin{aligned}
    \norm{a}^2&=\alpha\beta,\\
    \norm{b}^2&= \alpha\gamma,\\
     a^\inv\boxdot b&=\alpha c,
    \end{aligned}
\end{equation}
i.e.,
$$\q=
\begin{pmatrix}
    \alpha & a & b \\
    a^\inv &\frac1\alpha \norm{a}^2 & \frac1\alpha a^\inv\boxdot b \\ 
    b^\inv & \frac1\alpha b^\inv \boxdot a & \frac1\alpha \norm{b}^2
    \end{pmatrix}$$
and $\alpha=\alpha^2+\norm{a}^2+\norm{b}^2$ by \eqref{eq:q2=q}.

So, a minimal projection having a nonzero element at the place $11$ must have this form. As explained above, any minimal projection is either of this form or may be obtained from  a projection of this form by applying the automorphism $U_{1,2}$ or $U_{1,3}$ from Lemma~\ref{L:exchanging}. Now the necessity easily follows.
\end{proof}

\begin{proof}[Proof of the sufficiency.] Assume that $\q\in C_6$ has one of the three forms. We will show that it is a minimal projection. Since the first and the second forms could be transformed to the third one by applying  the automorphism $U_{1,3}$ or $U_{1,2}$ from Lemma~\ref{L:exchanging} and such an automorphism preserves minimal projections, we may restrict ourselves to the elements of the third form.

We therefore assume that
$$\q=\begin{pmatrix}
    \alpha & a & b \\
    a^\inv &\frac1\alpha \norm{a}^2 & \frac1\alpha a^\inv\boxdot b \\ 
    b^\inv & \frac1\alpha b^\inv \boxdot a & \frac1\alpha \norm{b}^2
    \end{pmatrix},$$  where $\alpha\in\er\setminus\{0\}$, $a,b\in\O_R$ and $\alpha=\alpha^2+\norm{a}^2+\norm{b}^2$.
    
We first show that $\q$ is a projection. Clearly $\q^*=\q$. To prove that $\q\boxdot\q=\q$ we need to verify equalities \eqref{eq:q2=q}. The first one follows from the very assumption of $\q$.
The second one is
$$\frac1\alpha\norm{a}^2=\norm{a}^2+\frac1{\alpha^2}\norm{a}^4+\frac1{\alpha^2}\norm{a^\inv\boxdot b}^2.$$
Observe that
$$\norm{a^\inv\boxdot b}^2=\ip{a^\inv\boxdot b}{a^\inv\boxdot b}=\ip{a\boxdot(a^\inv \boxdot b)}{b}=\norm{a}^2\norm{b}^2,$$
where we used Proposition~\ref{P:octonions}$(ii)$ and Proposition~\ref{P:real octonions}$(ii)$. Thus the second equality is equivalent to
$$\frac1\alpha\norm{a}^2=\norm{a}^2+\frac1{\alpha^2}\norm{a}^4+\frac1{\alpha^2}\norm{a}^2\norm{b}^2,$$
which is the $\frac{\norm{a}^2}{\alpha^2}$-multiple of the first equality. 

Similarly, the third equality is the $\frac{\norm{b}^2}{\alpha^2}$-multiple of the first one.

Let us look at the fourth equality. It reads
$$a=\alpha a+\frac1{\alpha}\norm{a}^2 a + b\boxdot \left(\frac 1\alpha b^\inv\boxdot a\right).$$
Let us compute the right-hand side:
$$\begin{aligned}
\alpha a+\frac1{\alpha}\norm{a}^2 a + b\boxdot \left(\frac 1\alpha b^\inv\boxdot a\right)&=
\alpha a+\frac1{\alpha}\norm{a}^2 a +\frac1\alpha\norm{b}^2a\\&=
\frac1\alpha(\alpha^2+\norm{a}^2+\norm{b}^2)a=a
\end{aligned}$$
by the first equality.
Similarly,
$$\begin{aligned}
\alpha b+a\boxdot\left(\frac1\alpha a^\inv\boxdot b\right)+\frac1{\alpha}\norm{b}^2 b&=\alpha b+\frac1\alpha \norm{a}^2 b+\frac1{\alpha}\norm{b}^2 b
\\&=\frac1\alpha(\alpha^2+\norm{a}^2+\norm{b}^2)b=b
\end{aligned}$$
and
$$\begin{aligned}
a^\inv\boxdot b +\frac1{\alpha}\norm{a}^2 \cdot \frac1\alpha a^\inv\boxdot b&+
\frac1{\alpha}\norm{b}^2 \cdot \frac1\alpha a^\inv\boxdot b\\&=\frac1{\alpha^2}(\alpha^2+\norm{a}^2+\norm{b}^2)(a^\inv\boxdot b)
\\&= \frac1\alpha a^\inv\boxdot b,
\end{aligned}$$
so the fifth and sixth equalities hold as well, which completes the proof that $\q$ is a projection.

 Next we are going to show that $\q$ is minimal. Clearly $\q\ne0$. Since $C_6$ has rank three, to show that there is a nonzero projection $\r$ with $\r\perp \q$ and $\q+\r\ne\1$. We will distinguish four cases:

Case 1: $a=b=0$. Then
$$\q=\begin{pmatrix}
    1&0&0\\0&0&0\\0&0&0
\end{pmatrix},$$
hence, for example, the choice of
$$\r=\begin{pmatrix}
    0&0&0\\0&1&0\\0&0&0
\end{pmatrix}$$
works.

Case 2: $a\ne0$ and $b=0$. Then the third row and the third column of $\q$ are zero, hence the choice of
$$\r=\begin{pmatrix}
    0&0&0\\0&0&0\\0&0&1
\end{pmatrix}$$
works.

Case 3: $a=0$ and $b\ne0$. Then the second row and the second column of $\q$ are zero, hence the choice of
$$\r=\begin{pmatrix}
    0&0&0\\0&1&0\\0&0&0
\end{pmatrix}$$
works.

Case 4: $a\ne 0$ and $b\ne 0$. Set 
$$\r=\begin{pmatrix}
\frac{\norm{a}^2}{\alpha^2+\norm{a}^2} & -\frac{\alpha a}{\alpha^2+\norm{a}^2} & 0 \\
-\frac{\alpha a^\inv}{\alpha^2+\norm{a}^2} & \frac{\alpha^2}{\alpha^2+\norm{a}^2}&0 \\0&0&0
\end{pmatrix}$$
By the above $\r$ is a projection. Indeed,
$$\left(\frac{\norm{a}^2}{\alpha^2+\norm{a}^2}\right)^2+ \frac{\alpha^2 a\boxdot a^\inv}{(\alpha^2+\norm{a}^2)^2}=\frac{\norm{a}^2(\alpha^2+\norm{a}^2)}{(\alpha^2+\norm{a}^2)^2}=\frac{\norm{a}^2}{\alpha^2+\norm{a}^2}
$$
and the second row is made from the first row by muliplying by
$$\frac{\alpha^2+\norm{a}^2}{\norm{a}^2}\cdot\left(-\frac{\alpha a^\inv}{\alpha^2+\norm{a}^2}\right)=-\frac{\alpha a^\inv}{\norm{a}^2}$$
from the left.

Moreover, $\q\perp \r$. Indeed,
$$\r\boxdot \q=\left(\begin{smallmatrix}
\frac{\alpha\norm{a}^2}{\alpha^2+\norm{a}^2}-
\frac{\alpha a\boxdot a^\inv}{\alpha^2+\norm{a}^2} & 
\frac{\norm{a}^2a}{\alpha^2+\norm{a}^2}-
\frac{\alpha a}{\alpha^2+\norm{a}^2}\cdot \frac{a^\inv\boxdot a}\alpha &
\frac{\norm{a}^2b}{\alpha^2+\norm{a}^2}-
\frac{\alpha a}{\alpha^2+\norm{a}^2}\boxdot \frac{a^\inv\boxdot b}\alpha
\\
-\frac{\alpha^2 a^\inv}{\alpha^2+\norm{a}^2} + \frac{\alpha^2 a^\inv}{\alpha^2+\norm{a}^2} & -\frac{\alpha a^\inv\boxdot a}{\alpha^2+\norm{a}^2} + \frac{\alpha^2}{\alpha^2+\norm{a}^2}\cdot\frac1\alpha a^\inv\boxdot a &  -\frac{\alpha a^\inv\boxdot b}{\alpha^2+\norm{a}^2} + \frac{\alpha^2}{\alpha^2+\norm{a}^2}\cdot\frac1\alpha a^\inv\boxdot b
\\ 0 & 0 &0
\end{smallmatrix}\right),$$
which equals $0$. Thus also $\q\boxdot \r=0$, hence $\q\circ \r=0$, i.e., $\r\perp \q$.
Finally, clearly $\q+\r\ne\1$ (as $b\ne0$).

This completes the proof.
\end{proof}

\section{Proofs of the results on unitaries in $C_6$}\label{sec:proofs}

This section is devoted to the synthesis of the results from previous sections which leads to proving Theorem~\ref{T:product} and Theorem~\ref{T:simultaneous biq}. We start with the second theorem.

\begin{proof}[Proof of Theorem~\ref{T:simultaneous biq}.]
Fix two unitary elements $\uu,\e\in C_6$. We will describe how to find a Jordan $*$-automorphism of $C_6$ with the required properties.  Note that the composition of Jordan $*$-automorphisms is again a Jordan $*$-automorphism. So, we will do the proof in several steps.

At first observe that by Theorem~\ref{T:spectral decomposition of unitary}$(iii)$ we may  assume, without loss of generality, that $\e$ is a diagonal matrix.

Next we apply Theorem~\ref{T:spectral decomposition of unitary}$(i)$ to $\uu$ and find complex units $\alpha_1,\alpha_2,\alpha_3$ and minimal projections $\q_1,\q_2,\q_3\in C_6$ such that $\uu=\alpha_1\q_1+\alpha_2\q_2+\alpha_3\q_3$. Since $\q_1+\q_2+\q_3=\1$, we may assume (up to relabelling the projections) that $\q_1$ has a nonzero element at the place $11$. Hence, by Theorem~\ref{T:minimal projections} we have
$$\q_1=\begin{pmatrix}
    \alpha & a & b \\
    a^\inv &\frac1\alpha \norm{a}^2 & \frac1\alpha a^\inv\boxdot b \\ 
    b^\inv & \frac1\alpha b^\inv \boxdot a & \frac1\alpha \norm{b}^2
    \end{pmatrix},$$  where $\alpha\in\er\setminus\{0\}$, $a,b\in\O_R$ and $\alpha=\alpha^2+\norm{a}^2+\norm{b}^2$.

By Proposition~\ref{P:OR-transform}$(a)$ there is an asymmetric triple isomorphism $S_1$ of $\O_R$ such that $S_1(a)\in\er$. It can be extended to a hermitian asymmetric triple isomorphism of $\O$, let us denote it still by $S_1$. Up to applying the automorphism $\widetilde{S_1}$ provided by Proposition~\ref{P:automorphisms-C6} we may now assume that $a\in\er$ (note that such an automorphism preserves diagonal matrices).

By Proposition~\ref{P:OR-transform}$(b)$ there is an automorphism $S_2$ of $\O_R$ such that $S_2(b)\in\span\{e_0,e_1\}$. It can be extended to an automorphism of $\O$, let us denote it still by $S_2$. Note that $S_2(1)=1$ and $S_2$ preserves $\er$. Hence,  up to applying the automorphism $\widetilde{S_2}$ provided by Proposition~\ref{P:automorphisms-C6} we may now assume that $b\in\span\{e_0,e_1\}$.

Summarizing, up to now we have reduced the general situation to the case when $\e$ is a diagonal matrix and the entries of $\q_1$ are in $\span\{e_0,e_1\}$.

We continue by proving that we may further assume without loss of generality that the entries of $\q_2$ are quaternions. By Theorem~\ref{T:minimal projections} we may distinguish three cases:

Case 1: $\q_2=\begin{pmatrix}
    0&0&0\\0&0&0\\0&0&1
\end{pmatrix}$. Then no more transformation is needed as the entries of $\q_2$ are real numbers.

Case 2: $\q_2=\begin{pmatrix}
0 & 0 & 0\\0&\xi&x\\0&x^\inv&\frac1\xi\norm{x}^2
\end{pmatrix}$, where $\xi\in\er\setminus\{0\}$, $x\in\O_R$ and $\xi=\xi^2+\norm{x}^2$.

By Proposition~\ref{P:OR-transform}$(c)$ there is an automorphism $U$ of $\O_R$ such that $U(e_1)=e_1$ and $U(x)\in\span\{e_0,e_1,e_2\}$. Then $U$ may be extended to an automorphism of $\O$, we will denote it still by $U$. Now take the automorphism $\widetilde{U}_3$ provided by Corollary~\ref{cor:C6-automorphisms}. This automorphism preserves diagonal matrices and also $\q_1$. So, we may  assume without loss of generality that
the entries of $\q_2$ are in $\span\{e_0,e_1,e_2\}$, in particular, they are quaternions.

Case 3: $\q_2=\begin{pmatrix}
    \xi & x & y \\
    x^\inv &\frac1\xi \norm{x}^2 & \frac1\xi x^\inv\boxdot y \\ 
    y^\inv & \frac1\xi y^\inv \boxdot x & \frac1\xi \norm{y}^2
    \end{pmatrix}$,  where $\xi\in\er\setminus\{0\}$, $x,y\in\O_R$ and $\xi=\xi^2+\norm{x}^2+\norm{y}^2$.

        Recall that $\q_3=\1-\q_1-\q_2$ is a minimal projection.
  We have
$$\q_3=\1-\q_1-\q_2=\begin{pmatrix}
    1-\alpha-\xi & -a-x & -b-y \\
    -a^\inv-x^\inv &1-\frac1\alpha\norm{a}^2-\frac1\xi \norm{x}^2 &-\frac1\alpha a^\inv\boxdot b- \frac1\xi x^\inv\boxdot y \\ 
    -b^\inv-y^\inv & -\frac1\alpha b^\inv\boxdot a-\frac1\xi y^\inv \boxdot x & 1-\frac1\alpha\norm{b}^2-\frac1\xi \norm{y}^2
    \end{pmatrix}.$$
We distinguish two subcases.

Subcase 3.1: $\alpha+\xi=1$. Then the first row and the first column of $\q_3$ are zero, in particular $x=-a$ and $y=-b$. Thus $x\in\er$ and $y\in\span\{e_0,e_1\}$, so the entries of $\q_2$ are in $\span\{e_0,e_1\}$.

Subcase 3.2: $\alpha+\xi<1$. Then we deduce from Theorem~\ref{T:minimal projections} that the following equalities hold (we use that $a\in\er$):
$$\begin{aligned}
1-\alpha-\xi&=(1-\alpha-\xi)^2+\norm{a+x}^2+\norm{b+y}^2,\\
1-\frac1\alpha\norm{a}^2-\frac1\xi \norm{x}^2&=\frac{1}{1-\alpha-\xi}\norm{a+x}^2,\\
-\frac1\alpha a b- \frac1\xi x^\inv\boxdot y&=\frac{1}{1-\alpha-\xi}(a+x^\inv)\boxdot (b+y),\\
1-\frac1\alpha\norm{b}^2-\frac1\xi \norm{y}^2&=\frac{1}{1-\alpha-\xi}\norm{b+y}^2.
\end{aligned}$$

Let us look at the third equality. By standard algebraic manipulation we get
$$\begin{aligned}
-(\tfrac1\alpha-1-\tfrac\xi\alpha)a b -(\tfrac1\xi-\tfrac\alpha\xi-1) x^\inv \boxdot y&=a b+x^\inv\boxdot b+a y +x^\inv\boxdot y,\\
-\tfrac{1-\xi}\alpha a b -\tfrac{1-\alpha}\xi x^\inv \boxdot y&=x^\inv\boxdot b+a y,\\
-a(\frac{1-\xi}{\alpha}b+y)&=x^\inv\boxdot(b+\frac{1-\alpha}{\xi}y).
\end{aligned}$$

Now there are two possibilities:

Sub-subcase 3.2.1: $b+\frac{1-\alpha}\xi y=0$. Then $y=-\frac{\xi}{1-\alpha}b$ (note that $\alpha<1$), hence $y\in\span\{e_0,e_1\}$. By Proposition~\ref{P:OR-transform}$(c)$ there is an automorphism $U$ of $\O_R$ preserving $e_1$ such that $U(x)\in\span\{e_0,e_1,e_2\}$. Then $U$ may be extended to an automorphism of $\O$, denoted still by $U$. Up to applying the automorphism $\widetilde{U}$ provided by Proposition~\ref{P:automorphisms-C6} we may assume that $x\in\span\{e_0,e_1,e_2\}$, thus the entries of $\q_2$ are quaternions.

Sub-subcase 3.2.2: $b+\frac{1-\alpha}\xi y\ne 0$. By Proposition~\ref{P:real octonions} we then deduce that
$$x^\inv=-\frac{1}{\norm{b+\frac{1-\alpha}\xi y}^2}a (\frac{1-\xi}{\alpha}b+y)\boxdot (b+\frac{1-\alpha}\xi y)^\inv.$$
By Proposition~\ref{P:OR-transform}$(c)$ there is an automorphism $U$ of $\O_R$ preserving $e_1$ such that $U(y)\in\span\{e_0,e_1,e_2\}$. Then $U$ may be extended to an automorphism of $\O$, denoted still by $U$. Up to applying the automorphism $\widetilde{U}$ provided by Proposition~\ref{P:automorphisms-C6} we may assume that $y\in\span\{e_0,e_1,e_2\}$. By the above formula we deduce that $x^\inv$ (and hence $x$) is a quaternion. Thus the entries of $\q_2$ are quaternions.

Summarizing, we have proved that, up to applying a Jordan $*$-automorphism of $C_6$ preserving diagonal matrices, we may assume that the entries of $\q_1$ and $\q_2$ are quaternions. Thus the entries of $\q_3=\1-\q_1-\q_2$ are quaternions as well. Since $\uu$ is a linear combination of $\q_1,\q_2,\q_3$, its entries must be biquaternions.
This completes the proof.
\end{proof}

We continue by establishing connections between the mappings $\dt$ and $\dt_3$.

\begin{lemma}\label{L:rank two}
Let $\q\in H_3(\Ha_C)$ be a minimal projection. Then $\widehat{\q}$ is a matrix of rank two.
\end{lemma}

\begin{proof}
The projection $\q$ is clearly minimal also in $C_6$, so it has the form from Theorem~\ref{T:minimal projections} and its entries are quaternions. Let us distinguish the three cases:

Case 1: If $\q=\begin{pmatrix}
0 & 0 & 0\\0&0&0\\0&0&1
\end{pmatrix}$, then $\widehat{\q}$ has clearly rank two.

Case 2: Assume 
$$\q=\begin{pmatrix}
0 & 0 & 0\\0&\alpha&a\\0&a^\inv&\frac1\alpha\norm{a}^2
\end{pmatrix},$$
where $\alpha\in\er\setminus\{0\}$, $a\in\Ha$, and $\alpha=\alpha^2+\norm{a}^2$.
The third row is obtained by multiplying the second one by $\frac1\alpha a^\inv$ from the left. Hence, the fifth and sixth rows of $\widehat{\q}$ are linear combinations of the third and fourth rows.
Since the third and fourth rows are linearly independent (as $\alpha\ne0$), $\widehat{\q}$ is of rank two.

Case 3: Assume
$$\q=
  \begin{pmatrix}
    \alpha & a & b \\
    a^\inv &\frac1\alpha \norm{a}^2 & \frac1\alpha a^\inv\boxdot b \\ 
    b^\inv & \frac1\alpha b^\inv \boxdot a & \frac1\alpha \norm{b}^2
    \end{pmatrix},$$
  where $\alpha\in\er\setminus\{0\}$, $a,b\in\Ha$, $\alpha=\alpha^2+\norm{a}^2+\norm{b}^2$.
    Then the second row is obtained by multiplying the first one by $\frac1\alpha a^\inv$ from the left and the third row is obtained by multiplying the first one by $\frac1\alpha b^\inv$ from the left. So, each row of $\widehat{\q}$ is a linear combination of the first two.
\end{proof}

\begin{lemma}\label{L:dt=dt3}
Let $\uu\in H_3(\Ha_C)$ be a unitary element. Then $\dt\uu=\dt_3\uu$.
\end{lemma}

\begin{proof}
It follows from Theorem~\ref{T:spectral decomposition of unitary}$(i)$ that there are complex units $\alpha_1,\alpha_2,\alpha_3$ and minimal projections $\p_1,\p_2,\p_3\in C_6$ such that $\uu=\alpha_1\p_1+\alpha_2\p_2+\alpha_3\p_2$. It follows from the proof of Theorem~\ref{T:spectral decomposition of unitary}$(i)$ that in this case $\p_1,\p_2,\p_3\in H_3(\Ha_C)$. Clearly the expression 
$$\widehat{\uu}=\alpha_1\widehat{\p_1}+\alpha_2\widehat{\p_2}+\alpha_3\widehat{\p_3}$$ coincides with the spectral decomposition of $\widehat{\uu}$ in $M_6$.

By Lemma~\ref{L:rank two} projections $\widehat{\p_j}$ are of rank two. Now it easily follows that
$$\dt \uu=\alpha_1\alpha_2\alpha_3=\dt_3\uu,$$
where the first equality follows from the very definition of $\dt$ and the second one follows from Theorem~\ref{T:determinant}$(ix)$.
\end{proof}

\begin{lemma}\label{L:dte=dt3e}
Let $\e,\uu\in H_3(\Ha_C)$ be unitary elements. Then $\dt_{\e}\uu=\dt_{3,\e}\uu$.
\end{lemma}

\begin{proof}
We have $\uu=\alpha_1\e_1+\alpha_2\e_2+\alpha_3\e_3$, where $\alpha_1,\alpha_2,\alpha_3$ are complex units and $\e_1,\e_2,\e_3$ are pairwise orthogonal minimal tripotents in $C_6$ satisfying $\e_j\le \e$ for each $j$. By the very definition we have
$$\dt_{\e}\uu=\alpha_1\alpha_2\alpha_3.$$
By Lemma~\ref{L:sqrt} there is a unitary element $\vv\in H_3(\Ha_C)$ such that $\vv^2=\e$. Then
$$T(\x)=\J{\vv^*}{\x^*}{\vv^*},\quad \x\in C_6$$
is a triple automorphism of $C_6$ (by Lemma~\ref{L:shift}) which preserves $H_3(\Ha_C)$ and satisfies $T(\e)=\1$.
Then
 $$T\uu=\alpha_1T\e_1+\alpha_2T\e_2+\alpha_3T\e_3,$$
and $T(\e_j)\le T(\e)=\1$, so $T\e_j$ are (parwise orthogonal minimal) projections in $C_6$. Thus
$$\dt_{\e}\uu=\alpha_1\alpha_2\alpha_3=\dt T\uu =\dt_3 T\uu=\dt_{3,\e}\uu,$$
where the third equality follows from Lemma~\ref{L:dt=dt3}, while the fourth one is a consequence of the definition in Proposition \ref{P:product-dtn}.
\end{proof}     

Now we are ready to prove the remaining theorem:

\begin{proof}[Proof of Theorem~\ref{T:product}.]
Let $\e,\uu\in C_6$ be unitary elements. By Theorem~\ref{T:simultaneous biq} there is a Jordan $*$-automorphism $T$ of $C_6$ such that both $T(\e)$ and $T(\uu)$ belong to $H_3(\Ha_C)$. Then
$$\begin{aligned} \dt \uu& =\dt T(\uu)=\dt_3 T(\uu)= \dt_{3,T(\e)} T(\uu)\cdot \dt_3 T(\e)\\
&= \dt_{T(\e)} T(\uu)\cdot \dt T(\e)=\dt_{\e}\uu\cdot \dt\e.
\end{aligned}$$
The first equality is clear, the second one follows from Lemma~\ref{L:dt=dt3}. The third equality follows from Proposition~\ref{P:product-dtn}$(ii)$,
the fourth one follows from Lemma~\ref{L:dt=dt3} and Lemma~\ref{L:dte=dt3e}. The last equality easily follows from definitions.

This completes the proof.
\end{proof}

\section{Determinants of general elements of $C_6$}\label{sec:dt in C6}

We have introduced determinants of unitary elements in $C_6$ (see Section~\ref{sec:unitaries}) and also of  general $^\inv$-hermitian matrices of biquaternions (in Section~\ref{sec:dt-n}). These two notions of determinant coincide provided both are defined (by Lemmata~\ref{L:dt=dt3} and~\ref{L:dte=dt3e}). In this section we extend the definition of determinant to general elements of $C_6$. 

To this end we first observe that some of the results on unitaries hold also for self-adjoint elements (after some natural modifications).

The following lemma is an analogue of  Theorem~\ref{T:spectral decomposition of unitary}. Its proof is essentially the same, we only should use assertion $(c)$ from Proposition~\ref{P:spectral} instead of assertion $(b)$.

\begin{lemma}\label{L:spectral decomposition of s-a}
  Let $\a\in C_6$ be a self-adjoint element. 
  \begin{enumerate}[$(i)$]
      \item There are real numbers $\alpha_1,\alpha_2,\alpha_3$ and mutually orthogonal minimal projections $\p_1,\p_2,\p_3$ 
such that $\a=\alpha_1\p_1+\alpha_2\p_2+\alpha_3\p_3$.
\item The representation from $(i)$ is unique in the natural sense: The triple $(\alpha_1,\alpha_2,\alpha_3)$ is uniquely determined up to reordering. Further, for any real number $\alpha$ the sum $\sum_{\alpha_j=\alpha}\p_j$ is also uniquely determined.
\item There is a Jordan $*$-automorphism $T$ of $C_6$ such that $T(\a)$ is a diagonal matrix (with real entries).
  \end{enumerate}
\end{lemma}

The previous lemma enables us to define determinant, $\dt \a,$ of a self-adjoint element, $\a,$ of $C_6$ as the product $\alpha_1\alpha_2\alpha_3$. This is compatible with the general definition below.

There is a natural common roof of unitary and self-adjoint elements -- an element $\x\in C_6$ is said to be \emph{normal} if it is a linear combination of mutually orthogonal (minimal projections) or, equivalently, if the Jordan $*$-subalgebra generated by $\x$ and $\x^*$ is associative.

The next lemma is an analogue of Theorem~\ref{T:simultaneous biq} and the same proof works.

\begin{lemma}\label{L:simultaneous quat}
Let $\a,\b\in C_6$ be two normal elements. Then there is a Jordan $*$-automorphism $T:C_6\to C_6$ such that
\begin{enumerate}[$(i)$]
    \item $T(\a)$ is a diagonal matrix;
    \item The entries of $T(\b)$ are biquaternions.
\end{enumerate}
\end{lemma}

We now deduce an easy consequence.

\begin{cor}\label{cor:x-biq}
Let $\x\in C_6$. Then there is a Jordan $*$-automorphism $T:C_6\to C_6$ such that the entries of $T(\x)$ are biquaternions.
\end{cor}

\begin{proof}
Observe that $\x=\a+i\b$, where $\a,\b$ are self-adjoint (and hence normal) elements of $C_6$ and use Lemma~\ref{L:simultaneous quat}.
\end{proof}

In the next theorem we introduce determinant of a general element of $C_6$ and collect some basic properties of this notion.

\begin{thm}\label{T:dt in C6}
 Let $\x\in C_6$.
 \begin{enumerate}[$(a)$]
     \item There is a unitary element $\uu\in C_6$ such that $\x=\J{\uu}{\x}{\uu}$.
     \item Let $\uu\in C_6$ be a unitary element provided by assertion $(a)$. Then:
       \begin{enumerate}[$(i)$]
      \item There are real numbers $\alpha_1,\alpha_2,\alpha_3$ and mutually orthogonal minimal tripotents $\uu_1,\uu_2,\uu_3$ 
such that $\uu_j\le\uu$ for $j=1,2,3$ and $\x=\alpha_1\uu_1+\alpha_2\uu_2+\alpha_3\uu_3$.
\item The representation from $(i)$ is unique in the natural sense: The triple $(\alpha_1,\alpha_2,\alpha_3)$ is uniquely determined up to reordering. Further, for any real number $\alpha$ the sum $\sum_{\alpha_j=\alpha}\uu_j$ is also uniquely determined.
  \end{enumerate}
\item Let $\uu\in C_6$ be a unitary element provided by $(a)$. Let us fix the representation of $\x$ as in $(b)$. Set
$$\dt_{\uu}\x=\alpha_1\alpha_2\alpha_3.$$
In case $\x$ is unitary, this quantity equals to $\dt_{\uu}\x$ from Section~\ref{sec:unitaries}.

In case $\x$ and $\uu$ belong to $H_3(\Ha_C)$, we have $\dt_{\uu}\x=\dt_{3,\uu}\x$.
\item  The quantity
$$\dt\x=\dt_{\uu}\x\cdot\dt\uu$$
does not depend on the particualr choince of $\uu\in C_6$ provided by $(a)$. 

Moreover, if $\x$ is unitary, this quantity coincides with $\dt\x$ from Section~\ref{sec:unitaries}.

If $\x\in H_3(\Ha_C)$, then $\dt\x=\dt_3\x$.

\item Jordan $*$-automorphisms of $C_6$ preserves values of $\dt$.
 \end{enumerate}
\end{thm}

\begin{proof}
Let $r(\x)$ be the range tripotent of $\x$ (see, e.g., \cite[beginning of Section 2]{BurgosFerGarMarPe2008}).  Since $C_6$ has finite dimension, $r(\x)$ exists and belongs to the JB$^*$-subtriple of $C_6$ generated by $\x$.

Then $\x$ is a positive element in $(C_6)_2(r(\x))$, in particular 
$$\x=\J{r(\x)}{\x}{r(\x)}.$$
Let $\uu\in C_6$ be a unitary element such that $r(\x)\le\uu$. It exists, for example, by \cite[Proposition 3.4 and Lemma 3.2$(d)$]{Finite}. 
By Peirce arithmetic we easily deduce $$\x=\J{\uu}{\x}{\uu},$$
which completes the proof of assertion $(a)$.

Assertion $(b)$ follows from Lemma~\ref{L:spectral decomposition of s-a}
applied to the JB$^*$-algebra $(C_6)_2(\uu)$ (which is Jordan $*$-isomorphic to $C_6$).

Let us define $\dt_{\uu}\x$ as in $(c)$. Assume $\x$ is unitary. The decompositions of $\x$ introduced before Theorem~\ref{T:product} coincides with the decomposition from $(b)$,
so the two versions of $\dt_{\uu}\x$ coincide.

Next assume that both $\uu$ and $\x$ belong to $H_3(\Ha_C)$. Let $T:M_3(\Ha_C)\to M_3(\Ha_C)$ be a linear mapping with properties from Lemma~\ref{L:isomorphism e-1} (with $\uu$ in place of $\e$). Then (see Proposition~\ref{P:product-dtn})
$$\dt_{3,\uu}\x=\dt_3(T\x).$$
Further,
$$T\x=\alpha_1 T\uu_1+\alpha_2 T\uu_2+\alpha_3 T\uu_3.$$

Having in mind that $T$ is a $*$-isomorphism of $M_{n}(\Ha_C)$ equipped with the product $\cdot_{\uu}$ and the involution $^{{*}_{\uu}}$ onto $M_{n}(\Ha_C)$ equipped with the standard structure of a C$^*$-algebra, we deduce that $T\uu_1, T\uu_2$ and $T\uu_3$ are minimal projections in $H_3(\Ha_C)$, so the above formula is the spectral decomposition of the self-adjoint element $T\x$ in $M_3(\Ha_C)$. Now it follows from Lemma~\ref{L:rank two} and Theorem~\ref{T:determinant}$(ix)$ that
$$\dt_3(T\x)= \alpha_1\alpha_2\alpha_3 =\dt_{\uu}\x.$$
This completes the proof of assertion $(c)$.

Let us continue by assertion $(d)$. Fix two unitary elements $\uu,\vv\in C_6$ such that
$$\x=\J{\uu}{\x}{\uu}=\J{\vv}{\x}{\vv}.$$
We will show that
$$\dt_{\uu}\x\cdot\dt\uu=\dt_{\vv}\x\cdot\dt\vv.$$
To this end fix a triple automorphism $S$ of $C_6$ such that $S(\uu)=\1$. Clearly
$$\dt_{\uu}\x=\dt_{S(\uu)} S(\x)=\dt_{\1} S(\x)\mbox{ and }\dt_{\vv}\x=\dt_{S(\vv)} S(\x).$$
Moreover, by Corollary~\ref{cor:dt after triple-aut}
$$\dt \vv=\dt S^{-1}(S(\vv))= \dt S(\vv)\cdot\dt S^{-1}(\1)=\dt S(\vv)\cdot\dt\uu.$$
It follows that we may assume without loss of generality that $\uu=\1$, i.e., it is enough to prove that
$$\dt_{\1}\x=\dt_{\vv}\x\cdot\dt\vv$$
whenever $\x\in C_6$ is self-adjoint and $\vv\in C_6$ is a unitary element with $\x=\J{\vv}{\x}{\vv}$.

All these properties are preserved by Jordan $*$-automorphisms. By Lemma~\ref{L:simultaneous quat} we may assume that both $\x$ and $\vv$ belong to $H_3(\Ha_C)$. In this case assertion $(c)$ yields $\dt_{\1}\x=\dt_3\x$ and $\dt_{\vv}\x=\dt_{3,\vv}\x$,  Lemma~\ref{L:dt=dt3} says that $\dt \uu=\dt_3\uu$ and so we may conclude by Proposition~\ref{P:product-dtn}$(ii)$. 

If $\x$ is unitary, the coincidence of the two formulae for $\dt\x$ follows from $(b)$ and Theorem~\ref{T:product}.

Further, assume $\x\in H_3(\Ha_C)$. Then $r(\x)\in H_3(\Ha_C)$, so we may find $\uu$ in $H_3(\Ha_C)$ as well. Then we conclude by $(c)$, Lemma~\ref{L:dt=dt3} and Proposition~\ref{P:product-dtn}.

Finally, assertion $(e)$ is obvious.
\end{proof}

Let us collect some further properties of the mapping $\dt$ on $C_6$ defined by Theorem \ref{T:dt in C6}$(d)$.

\begin{prop}
Let $T:C_6\to C_6$ be a triple automorphism. Then
$$\dt T(\x)=\dt \x\cdot \dt T(\1), \mbox{ for all }\x\in C_6.$$
\end{prop}

\begin{proof}
Fix a unitary element $\uu\in C_6$ such that $\x=\J{\uu}{\x}{\uu}$. Then $T(\uu)$ is a unitary element and $T(\x)=\J{T(\uu)}{T(\x)}{T(\uu)}$. Therefore
$$\dt T(\x)=\dt_{T(\uu)}T(\x)\cdot \dt T(\uu)=
\dt_{\uu}\x \cdot \dt \uu\cdot \dt T(\1)=\dt\x\cdot \dt T(\1),$$
where we used Corollary~\ref{cor:dt after triple-aut}.
\end{proof}

\begin{prop}
Let $\x\in C_6$. The following assertions are equivalent.
\begin{enumerate}[$(1)$]
    \item $\dt\x\ne0$;
    \item $r(\x)$ is a unitary element;
    \item $\x$ is invertible in $C_6$.
\end{enumerate}
\end{prop}

\begin{proof}
Recall that $\x$ is a positive element in $(C_6)_2(r(\x))$ and $r(\x)$ is the smallest tripotent with this property. Hence, the spectral 
decomposition theorem (cf. Proposition~\ref{P:spectral}) says that
\begin{equation}\label{eq spectral resolution of x in 10.6} \x=\sum_{j=1}^k \alpha_j\uu_j,
\end{equation}
where the $\uu_j$'s are mutually orthogonal minimal tripotents with sum $r(\x)$ and $\alpha_j>0$ for each $j$.

Note that $k\le3$ and $k=3$ if and only if $r(\x)$ is unitary. Now it easily follows that $(1)\Leftrightarrow(2)$.

$(2)\Rightarrow(3)$ Assume that $r(\x)$ is unitary. Then $\x$ has the above form and $k=3$. Its inverse is
$$\y=\sum_{j=1}^3 \alpha_j^{-1}\uu_j^*.$$
Indeed,
$$\x\circ\y=\sum_{j=1}^3 \uu_j\circ\uu_j^*
=\sum_{j=1}^3\J {\uu_j}{\uu_j}{\1}=\sum_{j=1}^3\J {\uu_j}{r(\x)}{\1}=\J{r(\x)}{r(\x)}{\1}=\1.$$
Further,
$$\begin{aligned}
\x^2\circ \y&=\J{\y}{\1}{\J{\x}{\1}{\x}} 
=\J{\J{\y}{\1}{\x}}{\1}{\x}-\J{\x}{\J{\1}{\y}{\1}}{\x}+\J{\x}{\1}{\J{\y}{\1}{\x}}
\\&=\J{\y\circ\x}{\1}{\x}-\J{\x}{\y^*}{\x}+\J{\x}{\1}{\y\circ\x}\\&=\J{\1}{\1}{\x}-\J{\x}{\y^*}{\x}+\J{\x}{\1}{\1}=2\x-\J{\x}{\y^*}{\x}.
\end{aligned}$$
Since
$$\begin{aligned}
\J{\x}{\y^*}{\x}&=\J{\sum_{j=1}^3 \alpha_j\uu_j}{\sum_{j=1}^3 \alpha_j^{-1}\uu_j}{\sum_{j=1}^3 \alpha_j\uu_j}=\sum_{j=1}^3\alpha_j\J{\uu_j}{\uu_j}{\uu_j}\\&=\sum_{j=1}^3\alpha_j\uu_j=\x,\end{aligned}$$
we deduce that $\x^2\circ\y=\x$. So, $\y$ is the Jordan inverse of $\x$.

$(3)\Rightarrow(2)$ 
 Suppose that $k<3$ in \eqref{eq spectral resolution of x in 10.6}. Under this assumption there exists a minimal tripotent $\w\in C_6$ which is orthogonal to $\x$ (equivalently to $\uu_j$ for all $1\leq j\leq k$). It is known that in a Jordan algebra an element $\x$ is invertible if and only if the mapping $U_{\x}$ is (cf. \cite[Theorem 4.1.3]{Cabrera-Rodriguez-vol1}). By orthogonality $$U_{\x} (\w^*) = \{\x,\w,\x\} =0,$$ contradicting the invertibility of $U_{\x}$. Therefore $k=3,$ and hence $r(\x) = \sum_{j=1}^3 \uu_j$ is a unitary element in $C_6$ because the latter has rank $3$.

 Actually, even more is true -- the range tripotent of each invertible element in a unital JB$^*$-algebra, $B$, is a unitary element belonging to $B$ (cf. \cite[Remark 2.3]{JamPeSidTah2015}).
\end{proof}

\begin{prop}
Let $\x\in C_6$. Then $\lambda\mapsto\dt(\lambda\cdot\1-\x)$ is a complex polynomial of degree $3$ with coefficient $1$ at $\lambda^3$ and coefficient $-\dt\x$ at $\lambda^0$.
\end{prop}

\begin{proof}
If $\x\in H_3(\Ha_C)$, the statement follows from Theorem~\ref{T:determinant}$(viii)$,$(v)$ and Theorem~\ref{T:dt in C6}$(d)$. The general case may be reduced to this case using Corollary~\ref{cor:x-biq} and Theorem~\ref{T:dt in C6}$(e)$.
\end{proof}

\smallskip\smallskip

\textbf{Acknowledgements} The first author was supported by the project OPVVV CAAS 
 CZ.02.1.01/0.0/0.0/16\_019/0000778.    Third author partially supported by MCIN / AEI / 10. 13039 / 501100011033 / FEDER ``Una manera de hacer Europa'' project no.  PGC2018-093332-B-I00, Junta de Andaluc\'{\i}a grants FQM375, A-FQM-242-UGR18 and PY20$\underline{\ }$ 00255, and by the IMAG--Mar{\'i}a de Maeztu grant CEX2020-001105-M / AEI / 10.13039 / 501100011033.

\bibliography{hkp}\bibliographystyle{acm}

\end{document}